\documentclass[11pt]{amsart}
\usepackage[english]{babel}
\usepackage[usenames,dvipsnames]{color}  
\usepackage{array,mathpazo,amsmath,xcolor}
\usepackage{amssymb}
\usepackage{stmaryrd, mathtools}
\usepackage[shortlabels]{enumitem}
\usepackage[colorlinks=true,linkcolor=Cerulean, citecolor=LimeGreen, pagebackref]{hyperref} 

\usepackage[all]{xy}
\SelectTips{eu}{} 
\xyoption{curve}

\def\lto{{\longrightarrow}}
\def\into{{\hookrightarrow}}
\def\xto{\xrightarrow}
\def\onto{\twoheadrightarrow}

\newcommand{\calh}{{\mathcal H}}

\newcommand{\calr}{{\mathcal R}}

\newcommand{\CC}{{\mathbb C}}

\newcommand{\HH}{{\mathbb H}}

\newcommand{\RR}{{\mathbb R}}
\renewcommand{\SS}{{\mathbb S}}

\newcommand{\ZZ}{{\mathbb Z}}

\newcommand{\bfi}{{\mathbf i}}
\newcommand{\bfj}{{\mathbf j}}
\newcommand{\bfk}{{\mathbf k}}

\newcommand{\sfD}{\ensuremath{\mathsf D}}

\newcommand{\sfG}{\ensuremath{\mathsf G}}

\newcommand{\vp}{\varphi}

\DeclareMathOperator{\Aut}{Aut}
\DeclareMathOperator{\car}{char}

\DeclareMathOperator{\Cl}{\mathit{Cl}}

\DeclareMathOperator{\GL}{GL}

\DeclareMathOperator{\id}{id}
\DeclareMathOperator{\Imm}{Im}

\DeclareMathOperator{\Ker}{Ker}

\DeclareMathOperator{\Mat}{\mathsf{Mat}}

\DeclareMathOperator{\Or}{\mathsf {O}}

\DeclareMathOperator{\Pin}{\mathsf{Pin}}
\DeclareMathOperator{\PIN}{\mathsf{PIN}}

\DeclareMathOperator{\rank}{rank}
\DeclareMathOperator{\Refl}{Refl}

\DeclareMathOperator{\SL}{SL}
\DeclareMathOperator{\SO}{SO}
\DeclareMathOperator{\Spin}{\mathsf{Spin}}
\DeclareMathOperator{\SU}{SU}
\DeclareMathOperator{\su}{\mathfrak{su}}

\DeclareMathOperator{\Sym}{{\mathbb S}ym}

\DeclareMathOperator{\tr}{tr}

\DeclareMathOperator{\Un}{U}

\newcommand{{\sbullet}}{{\scriptstyle\bullet}}

\newcommand{\uzero}{\underline 0}
\newcommand{\uone}{\underline 1}

\swapnumbers
\theoremstyle{definition}
\newtheorem{defn}{Definition}[section]

\newtheorem{remark}[defn]{Remark}

\newtheorem{remarks}[defn]{Remarks}

\newtheorem{sit}[defn]{}
\newtheorem{example}[defn]{Example}

\theoremstyle{plain}

\newtheorem{proposition}[defn]{Proposition}
\newtheorem{theorem}[defn]{Theorem}

\newtheorem{lemma}[defn]{Lemma}
\newtheorem{cor}[defn]{Corollary}
\newtheorem{corollary}[defn]{Corollary}

\author[Ragnar-Olaf Buchweitz]{Ragnar-Olaf Buchweitz$^\dagger$}
\address{Dept.\ of Computer and Math\-ematical Sciences,
University  of Tor\-onto at Scarborough, 
1265 Military Trail, 
Toronto, ON M1C 1A4,
Canada}

\author{Eleonore Faber}
\address{
School of Mathematics, 
University of Leeds, 
LS2 9EJ Leeds, UK
}
\email{e.m.faber@leeds.ac.uk}

\author{Colin Ingalls}
\address{
School of Mathematics and Statistics,
Carleton University, 
Ottawa, ON K1S 5B6,
Canada}
\email{cingalls@math.carleton.ca}

\title{The magic square of reflections and rotations}
\dedicatory{Dedicated to the memories of two great 
geometers, H.M.S.~Coxeter 
and Peter Slodowy}

\begin{document}

\thanks{
R.-O. \!B.~and C.I.~were partially supported by their respective NSERC Discovery grants, and E.F.~ acknowledges support of the European Union's Horizon 2020 research and innovation programme under the 
Marie Sk{\l}odowska-Curie grant agreement No 789580. All three authors were supported by the Mathematisches Forschungsinstitut Oberwolfach through the Leibniz fellowship program in 2017, and E.F~ and C.I.~were supported by the Research in Pairs program in 2018. 
} 
\thanks{${}^\dagger$The first author passed away on November 11, 2017.}
\date{\today}

\begin{abstract}
The title refers to the square diagram of bijections between various classes of finite subgroups of classical groups. In this paper we survey how Coxeter's work implies a bijection between complex reflection groups of rank two and real reflection groups in $\Or(3)$. We also consider this magic square of reflections and rotations in the framework of Clifford algebras: we give an interpretation using (s)pin groups and explore these groups in small dimensions.
\end{abstract}

\maketitle


\section{Introduction}

The main objective of this paper is to explain  the following square diagram that bijectively relates 
finite rotation and reflection groups from Euclidean geometry to corresponding finite subgroups in 
$\GL(2,\CC)$. 
\begin{align*}
\xymatrix{
*++[F]\txt{Finite Subgroups\\
of $\SL(2,\CC)$\\ of even order\\
(up to conjugation)}&&&&&&
*++[F]\txt{Finite Reflection Subgroups\\
of $\GL(2,\CC)$ containing $-\mathbf 1$\\
(up to conjugation)}
\ar[llllll]+<55pt,0pt>_-{\small\txt{$- \,\cap\,\SL(2,\CC)$\\
1:1}}^-{}\\
\\
\\
*++[F]{\txt{Finite Subgroups\\
of $\SL(3,\RR)$\\
(up to conjugation)}} \ar@{-}[uuu]-<0pt,32pt>^-{\small\txt{1:1}}_-{ }&&&&&&
*++[F]{\txt{Finite Reflection Subgroups\\
of $\GL(3,\RR)$\\
(up to conjugation)}} 
\ar[llllll]+<55pt,0pt>_-{\small\txt{$- \,\cap\,\SL(3,\RR)$\\
1:1}}^-{ }\ar@{-}[uuu]-<-0pt,26pt>^-{\small\txt{1:1}}
}
\end{align*}

The bijections are well-known, although scattered in the literature. Our goal is to describe them, giving historical background and references, and finally to make them explicit, employing (s)pinor geometry and classical results about the Clifford algebra of Euclidean space $\RR^3$. In particular, the finite subgroups of $\SL(2,\CC)$ can also be seen as the finite subgroups of unit quaternions, which relates them to geometric algebras: the double covering of $\SO(3)$ by $\SU(2)$ yields the left vertical bijection in the square. Since $\SU(2)$ is the $\Spin$-group of the Clifford algebra of $\RR^3$, a natural question is, whether one can also interpret the right hand vertical map in the square as a covering of $\Or(3)$ by the  $\Pin$-group. This correspondence is made explicit in this paper (see Theorem \ref{Thm:main}). We explain all ingredients needed for its proof in the following sections.  \\
This leads to the last part of the paper, where we consider the various definitions of the $\Pin$ group that are in use in the literature, see also this mathoverflow post of the first author \cite{Ragnar-overflow}. We point out differences between them for some  small dimensional examples over $\RR$.

The edges in the magic square are frequently used to jump between the classifications of finite subgroups of $\SL(2,\CC), \SL(3,\RR)$ on the one side, and certain finite groups generated by reflections on the other side. In particular, the rotational symmetries of the Platonic solids in $\RR^3$ correspond to finite subgroups of $\RR^3$ generated by reflections, i.e., the full symmetry groups of the platonic solids, on the other side.
The upper edge between finite subgroups of $\SL(2,\CC)$ and finite reflection subgroups of $\GL(2,\CC)$ containing $-{\bf 1}$ has been used in the McKay correspondence: the classical McKay correspondence connects the irreducible representations of a finite group $G \leqslant \SL(2,\CC)$ to the cohomology of the minimal resolution of the quotient singularity $\CC^2/G$, which is an ADE-surface singularity \cite{BuchweitzMFO,BFI-ICRA,SprinbergVerdier,KnoerrerGroups}. Moreover, the McKay correspondence for reflection groups relates the non-trivial irreducible representations of a finite  complex reflection group to certain maximal Cohen--Macaulay modules over its discriminant, where the discriminant is an ADE-curve singularity, see \cite{BFI} for more details.

In this paper, we first take the reader on a tour through the magic square of reflections and rotations: in Sections \ref{sec:notation} and \ref{Sec:magicsquare} we first define all necessary groups and state our main result (Theorem \ref{Thm:main}): an explicit description of the corners of the square in terms of the (S)pin groups of the standard Euclidean structure on $\RR^3$. We also provide extensive references and historic background. Then, in order to prove our statements, we introduce quaternions in Section \ref{Sec:Quaternions} and show in particular, how to interpret the unimodular unitary group $U(2)^\pm$ in terms of quaternions (Corollary \ref{Cor:quaternionsunimodular}). Section \ref{Sec:Clifford3} deals with the Clifford algebra of Euclidean space $\RR^3$ and in Section \ref{Sec:Pin3} we discuss the various definition of its Pin groups. Their precise relations are stated in Theorem \ref{Thm:3Pindifferences}. We conclude Section \ref{Sec:Reflections} with the proof of Theorem \ref{Thm:main}. In the last part of the paper, Section \ref{Sec:CliffordGeneral}, we consider general Clifford algebras and reflections and how they fit into a square diagram (see Theorem \ref{Thm:squaregeneral}). Finally, we consider some small dimensional examples of real Clifford algebras.

\section{The groups} \label{sec:notation}

We will be concerned with real and complex rotations, so $K=\RR$ or $\CC$. Recall the following groups:

\begin{itemize}
\item[$\GL(n,K)$] is the {\em general linear group\/} of invertible $n\times n$ 
matrices with entries from the field $K$, for $n\geqslant 1$ a natural number.
\item 
[$\Un(n)$] is the {\em unitary group\/} of complex $n\times n$ matrices defined as
\[
\Un(n) = \{A\in\GL(n,\CC)\mid A\cdot \overline A^{T} = \mathbf 1\}\,.
\] 
Throughout, $B^{T}$ denotes the transpose of the matrix $B$, and $\overline B$ is the 
complex conjugate matrix, where every entry from $B$ is replaced by its complex conjugate.
The matrix $\overline B^{T} = \overline {(B^{T})} = (\overline B)^{T}$ is also called the {\em Hermitian transpose\/}
of $B$ and is often denoted $B^{\dagger}$.
\item 
[$\SU(n)$] is the {\em special unitary group\/} of complex $n\times n$ matrices,
\[
\SU(n) = \{A\in\Un(n)\mid \det A = 1\}\,.
\]
As we will recall below, the group $\SU(2)$ is isomorphic to $\Spin(3)$, the spinor group
of Euclidean $3$--space with either its positive or its negative definite quadratic form.
\item 
[$\Or(n)$] is the {\em orthogonal group\/} of real $n\times n$ matrices,
\[
\Or(n) = \{A\in\GL(n,\RR)\mid A\cdot A^{T} = \mathbf 1\}\,.
\]
By the {\sc Cartan--Dieudonn\'e} Theorem, each element of $\Or(n)$ is the product of at most $n$ 
{\em reflections}. 
\item 
[$\SO(n)$] is the {\em special orthogonal group},
\[
\SO(n) = \{A\in\Or(n)\mid \det A = 1\}\,.
\]
Each element of $\SO(n)$ is a product of rotations, where a {\em rotation\/} is the product of 
two reflections from $\Or(n)$. In particular, every element in $\SO(3)$ is a rotation, which is {\sc Euler}'s
Theorem.
\item[$\mu_{n}$] is the multiplicatively written cyclic group of order $n$. Its additive incarnation is denoted
$\ZZ{/}n\ZZ$, and, generically, $\mathsf C_{n}$ stands for any cyclic group of order $n$. Instead of $\mu_{2}$
we also write $\{\pm 1\}$.
\end{itemize}

With details to follow below, we  single out the following groups, that will be an important part of relating the standard groups with (s)pinor geometry  and classical results on the structure of the Clifford algebra of Euclidean space $\RR^3$. 

\begin{defn}
For $n\geqslant 1$ an integer, set 
\[
\Un(n)^{\pm 1} =
\{A\in \GL(n,\CC)\mid A\cdot \overline A^{T} = \mathbf 1, \  \det A = \pm 1\}\leqslant \Un(n) \leqslant\GL(n,\CC)\,,
\] 
and call this subgroup of $\GL(n,\CC)$ the {\em unimodular%
\footnote{According to {\sc MathSciNet}, the terminology ``unimodular unitary group'' was used until about 1975 
to denote what we now generally call the special unitary group $\SU(n)$. There is some debate in general;
see the Wikipedia discussion relating to the page ``Unimodular matrix''; whether ``unimodular'' should pertain to 
$\det=1$ or to $\det =\pm1$. We will encounter the same ambiguity below when discussing Pin groups.
Because the use of ``unimodular unitary group'' to denote $\SU(n)$ was apparently abandoned after 1975,
we feel justified to reuse the term as proposed here.}
unitary group}. 

\end{defn}
Clearly, $\Un(n)^{\pm 1}$ contains $\SU(n)$ as a subgroup of index $2$.

With $\Mat_{n\times n}(\CC)$ the ring of $n\times n$ matrices over the complex numbers, one has obvious
inclusions
\[
\SU(n)\leqslant \Un(n)^{\pm1}\leqslant\Un(n)\leqslant \GL(n,\CC)\subset \Mat_{n\times n}(\CC)\,,
\]
where the first three inclusions are group homomorphisms, the last one an inclusion of multiplicative monoids
whose image is $\Mat_{n\times n}(\CC)^{*}=\GL(n,\CC)$, the group of (multiplicatively) invertible matrices.

\section{The Magic Square of Rotations and Reflections} \label{Sec:magicsquare}

\subsection*{Reflections, Take One}
Recall the following.
\begin{defn}
For a field $K$ of characteristic not $2$, and $V$ a finite dimensional vector space over it,
a $K$--linear endomorphism $s: V\to V$ is a \emph{reflection\/} (of order $2$)  if there exist a 
$K$--linear form $\lambda:V\to K$ and a vector $a\in V$ such that $\lambda(a)=1$ and 
$s(v)= v-2\lambda(v)a$, for all $v\in V$. 

In particular, the vector $a$ and the linear form $\lambda$ are not zero and 
the vector $a$ spans the eigenspace of $\vp$ corresponding to the eigenvalue $-1$, with 
$H=\Ker\lambda\subset V$ the complementary hyperplane that is left fixed pointwise by $s$, 
thus, is the eigenspace of $s$ corresponding to the eigenvalue $1$. This hyperplane is called the \emph{mirror\/}
of $s$.
\end{defn}

Complementing $a$ through a basis of $H$ to a basis of $V$, with $n=\dim_{K}V$ the 
endomorphism $s$ becomes represented by the $n\times n$ matrix $A\in \Mat_{n\times n}(K)$ of
the form
\begin{align*}
A =\begin{pmatrix}
-1\\
&1\\
&&\ddots\\
&&&1
\end{pmatrix}\,,
\end{align*}
the empty space to be filled with zeros. If $s'$ is a second reflection with $a'$ spanning the 
$(-1)$--eigenspace and $H'$ its mirror, then there are automorphisms $\vp$ of $V$ aplenty that 
send $a$ to $a'$ and $H$ to $H'$. Each such automorphism yields $s'= \vp s\vp^{-1}$, whence any 
two reflections are conjugate in $\GL(V)$.

\begin{example}
A matrix $A\in\Mat_{2\times 2}(K)$ represents a reflection if, and only if, its characteristic polynomial
is $\det(t\mathbf 1 - A) = t^{2} -1$, equivalently, $\tr A = 0$ and $\det A=-1$, where $\mathbf 1$
is the appropriate identity matrix and $\tr A\in K$  is the {\em trace\/} of $A$. 
\end{example}

\subsection*{The Magic Square}
We will establish the magic square from the introduction, here is a version with more detail labels for the bijections:
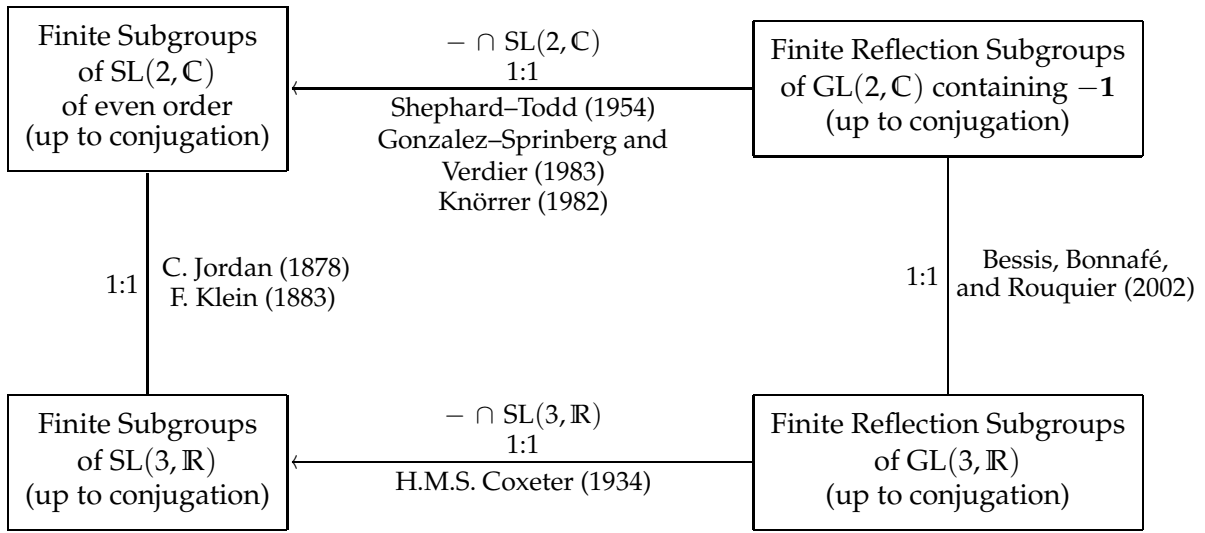
\begin{align*}
\xymatrix{
*++[F]\txt{Finite Subgroups\\
of $\SL(2,\CC)$\\ of even order\\
(up to conjugation)}&&&&&&
*++[F]\txt{Finite Reflection Subgroups\\
of $\GL(2,\CC)$ containing $-\mathbf 1$\\
(up to conjugation)}
\ar[llllll]+<55pt,0pt>_-{\small \txt{$- \,\cap\,\SL(2,\CC)$\\
1:1}}^-{\small\txt{Shephard--Todd (1954)\\
Gonzalez--Sprinberg and\\ Verdier (1983)\\
Kn¬\"orrer (1982)}}\\
\\
\\
*++[F]{\txt{Finite Subgroups\\
of $\SL(3,\RR)$\\
(up to conjugation)}} \ar@{-}[uuu]-<0pt,32pt>^-{\small\txt{1:1}}_-{\small\txt{ C. Jordan (1878)\\F. Klein (1883)}}&&&&&&
*++[F]{\txt{Finite Reflection Subgroups\\
of $\GL(3,\RR)$\\
(up to conjugation)}} 
\ar[llllll]+<55pt,0pt>_-{\small\txt{$- \,\cap\,\SL(3,\RR)$\\
1:1}}^-{\small\txt{H.M.S. Coxeter (1934)}}\ar@{-}[uuu]-<-0pt,26pt>^-{\small\txt{1:1}}_-{\small\txt{Bessis, Bonnaf\'e, \\ and  Rouquier (2002)}}
}
\end{align*}

\begin{remarks}
\begin{enumerate}[\rm(a)]
\item
Any finite (or, for that matter, any compact) subgroup in a Lie group is contained in a maximal compact subgroup. 
Thus, one may substitute $\SU(2)$ for $\SL(2,\CC)$ in the upper left box, $\Un(2)$ for $\GL(2,\CC)$ in the 
upper right box, or, going the other way, one may substitute $\SO(3)$ for $\SL(3,\RR)$ and  
$\Or(3)$ for $\GL(3,\RR)$, respectively, in the lower row.
\item The restriction to subgroups of even order in $\SL(2,\CC)$, respectively to
reflection groups in $\GL(2,\CC)$ containing $-\mathbf 1$ is not severe: Any finite subgroup of odd 
order in $\SL(2,\CC)$ is cyclic, and each such subgroup sits in a unique conjugacy class of reflection 
groups that do not contain $-\mathbf 1$. 
These are precisely the dihedral groups of type $\sfD_{2m+1}$, realized as reflection groups in 
$\GL(2,\CC)$; see also Prop.~\ref{Prop:dihedralodd} below.
\end{enumerate}
\end{remarks}

To fully appreciate the mathematics behind the magic square, recall the following classical facts.
\begin{theorem}[H.S.M.~Coxeter, {\cite[Theorem 18]{Cox1}}]
The inclusion of groups $j:\SO(3)\to \Or(3)$ sets up, via $j^{-1}(\sfG)= \sfG\cap \SO(3)$, 
a bijection between the conjugacy classes of finite reflection groups $\sfG\leqslant \Or(3)$ and the
conjugacy classes of finite subgroups of $\SO(3)$.
\end{theorem}

This yields the most elegant and efficient way to classify the finite subgroups of $\SO(3)$ by
first classifying the finite reflection groups in $\Or(3)$, a simple and very transparent task --- see
\cite{Cox3, ConwaySmith} for elegant detailed expositions.

\begin{sit}
Furthermore, two finite subgroups in $\SO(3)$ are conjugate if, and only if, 
they are isomorphic as abstract groups. The possible isomorphism types are:
\begin{enumerate}
\item the {\em cyclic groups\/} $\mathsf C_{n}$ of order $n\geqslant 1$,
\item the {\em dihedral groups\/} $\mathsf D_{n}$ of order $2n$, 
\item the {\em alternating group\/} $\mathsf A_{4}$ of order $12$, the rotational symmetries of the tetrahedron, 
\item the {\em symmetric group\/}  $\mathsf S_{4}$ of order $24$, the roational symmetries of the hexahedron(=cube) or octahedron, 
\item the {\em alternating group\/} $\mathsf A_{5}$ of order $60$, the rotational symmetries of the icosahedron or dodecahedron. 
\end{enumerate}

Regarding the corresponding finite reflection subgroups of $\Or(3)$, one just takes the full symmetry groups 
of the polyhedra with given rotational symmetries.
\end{sit}

\begin{sit}
The relationship between finite subgroups of $\SL(2,\CC)$, or its maximal compact subgroup $\SU(2)$ and
the finite subgroups of $\SO(3)$ is a consequence of Hamilton's discovery of the natural group homomorphism 
$\rho\colon \SU(2)\to \SO(3)$ that is surjective and has $\{\pm\mathbf 1\}$ as its kernel.

The next classical result then simply exploits the fact that $\SU(2)$, or $\SL(2,\CC)$ for that matter, contains
$-\mathbf 1$ as its sole element of order $2$. By Cauchy's theorem every finite group of even order has to 
contain an element of order $2$, whence the existence of Hamilton's homomorphism implies the following result.
\end{sit}

\begin{theorem}[C.~Jordan {\cite{Jordan}}, F.~Klein {\cite{KleinIcosaeder}}] 
\label{thm:JordanKlein}
The group homomorphism $\rho:\SU(2)\to \SO(3)$ sets up, via $\sfG\mapsto \rho^{-1}(\sfG)$, 
a bijection between the finite subgroups $\sfG\leqslant \SO(3)$ and the finite subgroups of $\SU(2)$
of {\em even\/} order, also called the {\em binary polyhedral groups}. 

Further, for every {\em odd\/} order, $\SU(2)$ contains a cyclic group of that order, mapped via $\rho$ 
isomorphically onto its image in $\SO(3)$.

It still holds for $\SU(2)$ that two finite subgroups are conjugate if, and only if, they
are isomorphic as abstract groups.
\end{theorem}

\begin{remark}
Indeed, C.~Jordan uses this result to classify the finite subgroups of $\SO(3)$ by first classifying the subgroups of
$\SU(2)$ that contain $-\mathbf 1$. To achieve the latter task, he uses the representation theory of Lie groups,
more precisely, the structure of stabilizers and normalizers of regular elements in such --- 
long before these notions were around. For a nice account of Jordan's approach, see \cite[\S 59]{Blichfeldt}.

F.~Klein goes the other way: he first famously classifies the finite subgroups in $\SO(3)$ as the groups of
rotational symmetries of polyhedra in $\RR^{3}$; see \cite[Ch.~19]{Armstrong} for a textbook account; and 
then lifts those groups to find all finite subgroups of $\SU(2)$.
\end{remark}

Our aim here is to make the remaining correspondences explicit in the setting of the above diagram, 
and to do so, we will employ (s)pinor geometry and classical results on the structure of the Clifford algebra of 
Euclidean space $\RR^{3}$. First recall the following:
for $n\geqslant 1$ an integer, 
\[
\Un(n)^{\pm 1} =
\{A\in \GL(n,\CC)\mid A\cdot \overline A^{T} = \mathbf 1, \ \det A = \pm 1\}\leqslant \Un(n) \leqslant\GL(n,\CC)\,,
\] 
is the unimodular unitary group. The various Pin groups for $\RR^3$ will be defined and discussed in Section \ref{Sec:Pin3}, and for general Clifford algebras in Section \ref{Sec:CliffordGeneral}.


Our main contribution is then the following identification, whose proof will be achieved in the following sections. 
\begin{theorem} \label{Thm:main}
\begin{enumerate}[\rm(1)]
\item The group $\Un(2)^{\pm 1}$ identifies naturally with the 
(geometric or big) {\em Pin group\/} $\PIN^{+}(3)$, associated to the standard Euclidean structure on 
$\RR^{3}$.
\item Hamilton's group homomorphism $\rho\colon \SU(2)\to \SO(3)$
extends via the twisted adjoint representation of $\PIN^{+}(3)$,
introduced by Atiyah, Bott, and Shapiro in \cite{ABS}, to a surjective group homomorphism
$\widetilde\rho\colon \Un(2)^{\pm 1}\to \Or(3)$ with kernel again equal to $\{\pm\mathbf 1\}$.
\item A matrix $A\in \Un(2)^{\pm 1}\leqslant\Mat_{2\times 2}(\CC)$ represents a (complex) reflection if, and only if,
$\widetilde\rho(A)$ represents a (real) reflection in $\Or(3)$.
\item The group $\Un(2)^{\pm 1}$ is generated by the complex reflections it contains.
\end{enumerate}
\end{theorem}

Restating this slightly, the backbone of the magic square above is the following statement that is, in essence, 
a reformulation of some results in \cite{BBR} --- and with proper reading one may argue that it is already 
contained in \cite{Cox3}.
\begin{theorem}
\label{thm:main}
There is a commutative diagram of exact sequences of groups
\[ \xymatrix{  
&1\ar[d]&1\ar[d]\\
&\{\pm\mathbf 1\}\ar@{=}[r]\ar[d]&\{\pm\mathbf 1\}\ar[d]\\
1\ar[r]&\SU(2)\ar[r]^-{i}\ar[d]_{\rho}&\Un(2)^{\pm1}\ar[r]^-{\det}\ar[d]_{\widetilde\rho}&\{\pm1\}\ar@{=}[d]\ar[r]&1\\
1\ar[r]&\SO(3)\ar[r]^-{j}\ar[d]&\Or(3)\ar[r]^-{\det}\ar[d]&\{\pm1\}\ar[r]&1 \\
&1&1\\ 
} \label{Diag:squareGroups}  \tag{$*$}
\]
where $i,j$ are the obvious inclusions, $\rho\colon \SU(2)\to \SO(3)$ is the universal covering of $\SO(3)$, 
the map famously discovered by Hamilton in terms of quaternions, and $\widetilde\rho$ identifies with 
the canonical covering of $\Or(3)$ by the Pin group.

Further, the two--dimensional complex representation afforded by the inclusion $\Un(2)^{\pm1}\leqslant \GL(2)$ 
and the three--dimensional representation of that group afforded by $\widetilde\rho$ are related in that 
$x\in\Un(2)^{\pm1}$ is a (complex) reflection of order $2$ in $\GL(2,\CC)$ if, and only if, $\widetilde\rho(x)$ 
is a (real) reflection in $\Or(3)$.
\end{theorem}

For the proof of these claims see below until Section \ref{Sec:Reflections}. 

\begin{remark}
The horizontal short exact sequence in the middle of above diagram could alternatively also be written as
\begin{align*}
\xymatrix{
1\ar[r]&\Spin(3)\ar[r]^-{i}&\PIN^{+}(3)\ar[r]^-{\det}&\{\pm1\}\ar[r]&1\\
}
\end{align*}
as we explain below.
\end{remark}

The group homomorphisms $\rho,\widetilde\rho$ of \eqref{Diag:squareGroups} can be defined, somewhat ad hoc, as follows.
\begin{sit}
Consider in $\Mat_{2\times 2}(\CC)$ the subspace
\begin{align*}
\su(2) =\left\{B\in \Mat_{2\times 2}(\CC)\mid {\overline{B}}^{T} = -B\,, \tr B=0
\right\}\,.
\end{align*}

In more fancy terms, $\su(2)$ consists of the {\em skew--Hermitian\/} (or {\em antihermitian}), {\em traceless\/} 
$2\times 2$ complex matrices. Explicitly, if $B$ is a matrix in $\su(2)$, then there are real numbers 
$\beta,\gamma,\delta\in \RR$ such that
\[
B= \begin{pmatrix}
\beta i & -\gamma -\delta i\\
\gamma -\delta i& -\beta i
\end{pmatrix}\,.
\]
Note that $\det B = \beta^{2}+\gamma^{2}+\delta^{2}\in \RR_{\geqslant 0}$ is the squared length of the
vector $(\beta,\gamma,\delta)\in\RR^{3}$.
In particular, $\su(2)$ is a $3$--dimensional real vector space and taking determinants yields the
standard Euclidean square norm on it.
\end{sit}

\begin{remark}
The reason for the notation is that $\su(2)$ is a (real) Lie algebra with respect to the usual commutator,
$[B,B'] =BB'-B'B$, 
%
and as such is the Lie algebra of the simply connected Lie group $\SU(2)$.
However, we will not need this additional insight here.
\end{remark}

\begin{sit}
If now $A\in \Un(2)^{\pm1}$ and $B\in\su(2)$, define
\begin{align*}
\widetilde\rho(A)(B) &= (\det A)ABA^{-1}\in \Mat_{2\times 2}(\CC)\,.
\end{align*}
It is not too painful to verify directly the following facts:
\begin{enumerate}[leftmargin=*, label=(\alph*)]
\item $\widetilde\rho(A)(B)\in\su(2)$ along with $B$, thus, $\widetilde\rho(A)$ is an $\RR$--linear automorphism 
of $\su(2)$, and 
\item $\widetilde\rho$ defines a group homomorphism from $\Un(2)^{\pm1}$ to 
$\Or(3)\leqslant \Aut_{\RR}(\su(2))\cong \GL(3,\RR)$. Indeed, $\det(\widetilde\rho(A)(B)) = \det(A)^2 \det B=  \det B$,
whence the Euclidean norm, given by $\det$ on $\su(2)$ is preserved.
\item The restriction of $\widetilde\rho$ to $\SU(2)$ yields $\rho:\SU(2)\to \SO(3)$. This group homomorphism
is in fact the {\em adjoint representation\/} of the Lie group $\SU(2)$ on its Lie algebra $\su(2)$, 
in that $\rho(A)(B) =ABA^{-1}$ for $A\in \SU(2), B\in \su(2)$.

From this point of view, the map $\widetilde \rho$ is the {\em twisted adjoint representation\/} of 
$\Un(2)^{\pm1}$ on its Lie algebra that happens to coincide with $\su(2)$, as $\SU(2)$ is the connected 
component of the identity element $\mathbf 1\in \Un(2)^{\pm1}$ in that Lie group. The twist is by the 
determinant character of $\Un(2)^{\pm 1}$ as the explicit form of the action shows.
\end{enumerate}
This defines the group homomorphisms $\rho,\widetilde \rho$ in Theorem \ref{thm:main}, but it does not yet
explain their surjectivity nor their kernels.
\end{sit}

Below we will put all this into its proper context of (s)pinor geometry.

These classical results, are ``complexified'' by Theorem \ref{thm:main}.
First, one can lift Coxeter's classification:

\begin{proposition}
The inclusion of groups $i:\SU(2)\to \Un(2)^{\pm 1}$ sets up, via $i^{-1}(\sfG)= \sfG\cap \SU(2)$, 
a bijection between the conjugacy classes of finite complex reflection groups 
$\sfG\leqslant \Un(2)^{\pm 1}\leqslant \GL(2,\CC)$ and the
conjugacy classes of finite subgroups of $\SU(2)$.
\end{proposition}

\begin{remark}
Any finite subgroup of $\GL(2,\CC)$ that is generated by reflections of order $2$
is conjugate to such a reflection group in $ \Un(2)^{\pm 1}$ and two such groups are conjugate in 
$ \Un(2)^{\pm 1}$ if, and only if, they are conjugate inside $\GL(2,\CC)$, see Lemma \ref{Lem:U2conjugate} below.

Thus, up to conjugation, this result classifies the finite subgroups in $\GL(2,\CC)$ generated by reflections 
of order $2$, and so recoups that part of the classification by Shephard--Todd \cite{STo} in a conceptual way.
\end{remark}

\begin{lemma} \label{Lem:U2conjugate}
Let $\sfG, \sfG' \subseteq \Un(2)^{\pm 1}$ be finite subgroups. Then $\sfG$ and $\sfG'$ are conjugate in $\Un(2)^{\pm 1}$ if and only if they are conjugate in $\GL(2,\CC)$.
\end{lemma}

\begin{proof}
If $\sfG, \sfG'$ are conjugate in $\Un(2)^{\pm 1}$, then they are clearly conjugate in $\GL(2,\CC)$. On the other hand, assume that $\sfG, \sfG'$ are conjugate in $\GL(2,\CC)$. 

Note that one can decompose any $g \in \GL(n,\CC)$ as $g=UDV^{-1}$, where $U, V \subseteq \Un(n)$ and $D$ is a diagonal matrix in $\GL(n,\RR)$ with entries in $\RR_{>0}$ (Singular Value Decomposition). However, given such a decomposition, we may first replace the matrix $U$ in this decomposition by $\frac{1}{\sqrt[n]{\det(U)}}U$, where $\sqrt[n]{\det U}$ is any of the $n$-th roots of the determinant (and similarly for $V$), to replace $g$ by a matrix with determinant equal to $1$ and in whose singular value decomposition the matrices $U,V$ are from the special unitary group. \\
Returning to the case at hand, $n=2$, we assume that $\sfG, \sfG'$ are conjugate in $\GL(2,\CC)$. That is, we may find a $g \in \GL(2,\CC)$ such that $\sfG=g\sfG'g^{-1}$. By the above, write $g=UDV^{-1}$ with $U, V \in \SU(2,\CC)$. Now the groups $\sfG_1:=U^{-1}\sfG U$ and $\sfG_2:=V^{-1}\sfG'V$ are still in $\Un(2)^{\pm1}$ and conjugate: $\sfG_1=D\sfG_2D^{-1}$, where $D=\begin{pmatrix} d_1 & 0 \\ 0 & d_2 \end{pmatrix}$ with $d_i >0$. Now any matrix $X$ in $\sfG_i$ is either of the form $\begin{pmatrix} \alpha & \beta \\ -\bar{\beta} & \bar{\alpha} \end{pmatrix}$ or $\begin{pmatrix} \alpha & \beta \\ \bar{\beta} & -\bar{\alpha} \end{pmatrix}$ for $\alpha, \beta \in \CC$ with $\alpha \bar{\alpha} + \beta \bar{\beta} = 1$. For any such $X \in \sfG_2$ of the first type one calculates
\[ DXD^{-1}= \begin{pmatrix} d_1 & 0 \\ 0 & d_2 \end{pmatrix} \begin{pmatrix} \alpha & \beta \\ -\bar{\beta} & \bar{\alpha} \end{pmatrix} \begin{pmatrix} \frac{1}{d_1} & 0 \\ 0 & \frac{1}{d_2} \end{pmatrix} = \begin{pmatrix} \alpha & \frac{d_1}{d_2}\beta \\ - \frac{d_2}{d_1} \bar{\beta} & \bar{\alpha} \end{pmatrix} \ , \]
which is in $\SU(2)$ if and only if, $\beta=0$, or $\frac{d_1}{d_2}=\frac{d_2}{d_1}$. In either case, $D$ commutes with $X$, thus, $D$ centralizes the group $\sfG_2$ (similarly for $X$ of the second type). This shows that $\sfG_1$ equals $\sfG_2$ , and hence the original groups $\sfG$ and $\sfG'$ have to be conjugate in $\Un(2)^{\pm1}$.
\end{proof}

\begin{sit} {\bf Historical Note.} 
It has been observed earlier by several authors that there is a strong relationship
between finite subgroups of $\SU(2)$ and complex reflection groups; see, for example, 
\cite{SprinbergVerdier, KnoerrerGroups}.
In those references, a reflection group $\widetilde\sfG\subseteq \GL(2,\CC)$ that contains a given finite
subgroup $\sfG\leqslant \SU(2)$ as subgroup of index $2$ was found by ``table look--up'': 
Shephard and Todd \cite{STo} compiled, in 1954, tables of all (conjugacy classes of)
finite groups generated by complex pseudo--reflections, and  this happens to imply the bijection between (conjugacy classes of) finite reflection 
groups in $\GL(2,\CC)$ and the (conjugacy classes of) binary polyhedral groups, the finite subgroups of $\SU(2)$.\\
In particular, the bijective correspondence at the right of the magic square, between finite reflection groups in $\GL(2,\CC)$
containing $-\mathbf 1$ and finite reflection groups in $\Or(3)$ remained mysterious for a long time.
Only in 2002, Bessis, Bonnaf\'e, and Rouquier \cite{BBR} explained this correspondence conceptually ---
not only for the case considered here, but with results for all finite subgroups of $\GL(n,\CC)$
generated by pseudo--reflections.
\end{sit}

The next result complexifies Theorem \ref{thm:JordanKlein}.

\begin{proposition} \label{Prop:dihedralodd}
The group homomorphism $\widetilde\rho:\Un(2)^{\pm1}\to \Or(3)$ sets up, via 
$\sfG\mapsto \widetilde\rho^{-1}(\sfG)$, 
a bijection between the finite subgroups $\sfG\leqslant \Or(3)$ generated by reflections
and the finite subgroups of $\Un(2)^{\pm1}$ that contain $-\mathbf 1$ and are generated by reflections.

Further, the finite subgroups of $\Un(2)^{\pm1}$ that do not contain $-\mathbf 1$ but are generated by 
reflections are the dihedral groups $\sfD_{2m+1}$, for $m\geqslant 0$, and those are 
mapped isomorphically via $\widetilde\rho$ onto their images in $\Or(3)$.
\end{proposition}

\begin{sit}
Finally we note the following application of Theorem \ref{thm:main} to root systems that is reminiscent
of a Theorem by E.~Witt. To formulate this, recall from \cite{HumphreysReflection}, say, 
that a finite subset $\calr\subset \RR^{n}$ of $n$--dimensional Euclidean space is a (normalized) 
{\em root system\/} if 
\begin{enumerate}[\rm (a)]
\item (Normalization) Each vector $v$ in $\calr$ is of length $1$,
\item For $v\in \calr$, one has $\RR v \cap \calr = \{v, -v\}$,
\item  For $v\in \calr$, the reflection $s_{v}\colon \RR^{n}\to \RR^{n}, s_{v}(w) = w - 2(v,w)v$,
with $(v,w)$ the usual positive definite scalar product on $\RR^{n}$, maps $\calr$ to itself
(observe that $(v,v)=1$ as $\calr$ is normalized).
\end{enumerate}
Normalization ensures that $\calr$ is indeed a subset of the unit sphere $\SS^{n-1}\subset\RR^{n}$.

Given a root system $\calr$, the group generated by the reflections $\{s_{v}\}_{v\in\calr}$ is naturally
a subgroup of the symmetric group $\Sym(\calr)$ on $\calr$, thus, this group constitutes a {\em finite\/}
reflection subgroup of the orthogonal group $\Or(n)$; see \cite[p. 7]{HumphreysReflection}.

Conversely, if $\sfG\leqslant \Or(n)$ is a finite reflection group, then the set $\calr$ of unit normal vectors 
to the mirrors of reflections in $\sfG$ forms a normalized root system.
\end{sit}

Now we are ready to recall Witt's Theorem. For more on quaternions, see Section \ref{Sec:Quaternions}.
\begin{theorem}[E.~Witt {\cite{Witt}}]
If one identifies $\SU(2)$ with the group $\HH_{1}$ of quaternions of norm $1$ (see \ref{unitquaternions=SU(2)}), and if 
$\sfG\leqslant \SU(2)\cong \HH_{1}$ is a subgroup of even order, then the elements of that group, 
viewed as vectors in the real vector space $\HH\cong\RR^{4}$ of all quaternions, form the root system 
of a finite reflection group in $\Or(4)$.
\end{theorem}
Witt then identified the elements of the binary icosahedral group in $\SU(2)$ as the root system $\mathbf H_{4}$
whose actual existence was not known until then.

Theorem \ref{thm:main} implies the following ``three-dimensional version'' of Witt's theorem.
\begin{theorem}
The set $\Refl(\Un(2)^{\pm1})$ of complex reflections in $\Un(2)^{\pm1}$ is
\begin{align*}
\Refl(\Un(2)^{\pm1}) &= \{A\in \Un(2)\mid \tr A =0\,, \det A = -1\}\\
&=\left\{
\begin{pmatrix}
\alpha&\beta+\gamma i\\
\beta-\gamma i&-\alpha
\end{pmatrix}\in\GL(2,\CC)\bigg|
\alpha,\beta,\gamma\in \RR, 
\alpha^{2}+\beta^{2}+\gamma^{2}=1
\right\}\,,
\end{align*}
thus, the set of reflections in $\Un(2)^{\pm1}$  is in natural bijection with the unit $2$--sphere 
$\SS^{2}\subset \RR^{3}$.

If $\sfG\leqslant \Or(3)$ is a reflection group with normalized root system $\calr$, then its preimage
$\widetilde\rho^{-1}(\sfG)\leqslant \Un(2)^{\pm1}$ is the complex reflection group generated by 
$\calr\subset \SS^{2}= \Refl(\Un(2)^{\pm1})$ as subset of $\Un(2)^{\pm1}$.
\end{theorem}

\begin{sit}
To summarize this section, one can replace the magic square from the beginning with the following
more precise diagram:
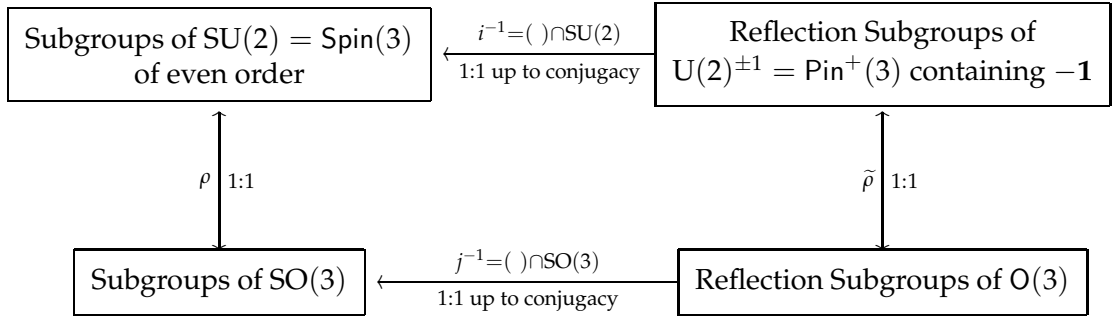
\begin{align*}
\xymatrix{
*++[F]\txt{Subgroups of $\SU(2)=\Spin(3)$\\ of even order}&&&
*++[F]\txt{Reflection Subgroups of \\ $\Un(2)^{\pm1}=\Pin^{+}(3)$ containing $-\mathbf 1$}
\ar[lll]+<85pt,0pt>_-{i^{-1} = (\ )\cap\SU(2)}^-{1:1\text{ up to conjugacy}}\\
\\
*++[F]{\txt{Subgroups of $\SO(3)$}} \ar@{<->}[uu]-<0pt,22pt>_-{1:1}^-{\rho}&&&
*++[F]{\txt{Reflection Subgroups of $\Or(3)$}} \ar@{<->}[uu]-<0pt,22pt>_-{1:1}^-{\widetilde\rho}
\ar[lll]+<60pt,0pt>_-{j^{-1} = (\ )\cap\SO(3)}^-{1:1\text{ up to conjugacy}}
}
\end{align*}
\end{sit}
To prove all the claims we have made so far, the essential ingredient is to understand the group
homomorphism $\widetilde\rho:\Un(2)^{\pm}\to \Or(3)$ thoroughly. We will first make everything explicit in terms
of quaternions and complex $2\times 2$--matrices and then recapture that description within the general theory of
Clifford algebras and their (s)pin groups.

\section{Quaternions} \label{Sec:Quaternions}
For beautiful treatments of the basic theory of quaternions, nothing beats the two references
\cite{EHH, Cox3} and we only excerpt a tiny bit from those.
\begin{sit}
If one views the quaternions $\HH\cong 1{\cdot} \CC\oplus \bfj{\cdot}\CC$ as $2$--dimensional (right) complex 
vector space, then multiplication from the left with elements from $\HH$ on itself is right $\CC$--linear and 
in the given basis $\{1,\bfj\}$ one obtains the (left) {\em regular representation\/} of $\HH$ on $\CC^{2}$, 
that is, the $\RR$--algebra homomorphism
\begin{align*}
\lambda\colon \HH=\RR1 \oplus\RR\bfi \oplus\RR \bfj\oplus\RR \bfk&\lto \Mat_{2\times 2}(\CC)\\
\alpha1 + \beta \bfi + \gamma \bfj +\delta \bfk&\mapsto
\begin{pmatrix}
\alpha + \beta i & -\gamma -\delta i\\
\gamma -\delta i& \alpha-\beta i
\end{pmatrix}\,.
\end{align*}
Indeed, for $\bfk\in \HH$, say, one obtains 
$\lambda(\bfk)(1) = \bfk= \bfj(-i)$ and $\lambda(\bfk)(\bfj) = \bfk\bfj = 1{\cdot}(-i)\in1{\cdot}\CC\subset  \HH$, 
so that the matrix of $\lambda(\bfk)$ is
\begin{align*}
\lambda(\bfk) =
\begin{array}{c|cc}
&1& \bfj\\
\hline
1&0&-i\\
\bfj&-i&0
\end{array}
\end{align*}
and similarly for $\lambda(1), \lambda(\bfi),\lambda(\bfj)$ and their (real) linear combinations.
\end{sit}

\begin{sit}
The map $\lambda$ is injective and yields an $\RR$--algebra isomorphism
\begin{align*}
\lambda\colon \HH\xto{\ \cong\ }
\calh :=   \left\{ \left. \begin{pmatrix}
w&-z\\
\overline z& \overline w
\end{pmatrix} \right| w,z\in \CC  \right\}\subset  \Mat_{2\times 2}(\CC)
\end{align*}
onto the real subalgebra $\calh$ of $\Mat_{2\times 2}(\CC)$. 
\end{sit}

\begin{sit}
To characterize $\calh$ as subset of $\Mat_{2\times 2}(\CC)$, with 
$x= \alpha1 + \beta \bfi + \gamma \bfj +\delta \bfk\in \HH$ write 
\[
\lambda(x) = \alpha\cdot\mathbf{1} + B\,,
\]
where $\mathbf 1$ is the $2\times 2$ identity matrix and $B$ is a {\em skew--Hermitian\/} (or {\em antihermitian}), {\em traceless\/} matrix, in that
\[
B= \lambda(\beta \bfi + \gamma \bfj +\delta \bfk)=\begin{pmatrix}
\beta i & -\gamma -\delta i\\
\gamma -\delta i& -\beta i
\end{pmatrix}
\]
satisfies
\[
\overline B^{T} = 
\begin{pmatrix}
-\beta i & \gamma +\delta i\\
-\gamma +\delta i& \beta i
\end{pmatrix} 
= - B\quad\text{and}\quad \tr B=0\,,
\]
with $\tr B$ denoting the {\em trace\/} of the matrix.
Conversely, if $\alpha\in\RR$ and $B$ is a complex $2\times 2$ matrix that is skew--Hermitian and traceless,
then $\alpha\mathbf{1} + B$ is in the image of $\lambda$.
\end{sit}

\begin{sit}
As already mentioned, the skew--Hermitian and traceless matrices form the real Lie algebra $\mathfrak{su}(2)$
in its standard representation on $\CC^{2}$, whence $\lambda(\HH)=\calh$ can be identified with
$\RR\cdot \mathbf 1\oplus \mathfrak{su}(2)\subset \Mat_{2\times 2}(\CC)$.
\end{sit}

\begin{sit}
Transporting things back to the quaternions, the {\em traceless\/} or {\em purely imaginary\/} or just 
{\em pure\/} quaternions $\HH^{0}=\RR i \oplus\RR j\oplus\RR k\subset \HH$ form a $3$--dimensional real 
vectorspace and $\lambda(\HH^{0})=\su(2)$.
\end{sit}

\begin{defn}
The {\em conjugate\/} of the quaternion $x= \alpha1 + \beta \bfi + \gamma \bfj +\delta \bfk$ is
\begin{align*}
\overline x &= \alpha1 - \beta \bfi - \gamma \bfj -\delta \bfk
\intertext{with the (square) {\em norm\/} of $x$ defined as}
N(x) &= x\overline x =\overline x x = \alpha^{2} + \beta^{2} + \gamma^{2} +\delta^{2} =\det \lambda(x)\in \RR_{\geqslant0}\,.
\end{align*}
Because $\det\lambda(x)=0$ only if $x=0$, this is the quickest way to see that $\HH\cong \calh$ is a 
{\em skew field}, also called a {\em division ring}.

Note that
\begin{align*}
\lambda(\overline x) = \begin{pmatrix}
\alpha -\beta i & \gamma +\delta i\\
-\gamma +\delta i& \alpha+\beta i
\end{pmatrix}
= (\overline{\lambda(x)})^{T}
\end{align*}
is the {\em Hermitian transpose\/} of the matrix $\lambda(x)$. Writing $\lambda(x)=\alpha\mathbf 1 + B$ as above,
then $\lambda(\overline x) = \alpha\mathbf 1 - B$, which brings out the analogy with complex conjugation.
\end{defn}

\begin{sit} \label{unitquaternions=SU(2)}
The norm restricts to a homomorphism from the multiplicative group $\HH^{*}$ of nonzero quaternions onto the
multiplicative group $(\RR_{>0},\cdot)$. It identifies via $\lambda$ with the restriction of the determinant 
homomorphism $\det\colon \GL(2,\CC)\to \CC^{*}$ to $\lambda(\HH^{*})$. The kernel of the norm
is the multiplicative group $(\HH_{1},\cdot)$ of quaternions of norm $1$, called {\em versors} (Latin for ``turners'') 
by Hamilton, and $\lambda$ induces a group isomorphism
\begin{align*}
\lambda|_{\HH_{1}}\colon \HH_{1}\xto{\ \cong\ } \SU(2)\leqslant \GL(2,\CC)
\end{align*}
onto the special unitary group
$\SU(2)\subset \calh$.
\end{sit}

\begin{sit}
To understand the embedding $\lambda\colon \HH \to \Mat_{2\times 2}(\CC)$ better, note that
for $A\in\calh$, the determinant of the matrix $Ai=iA$ will be $\det(Ai) = -\det A \in \RR_{\leqslant 0}$,
and therefore only the zero matrix simultaneously belongs to $\calh$ and $\calh i$, leading immediately to 
the following result.
\end{sit}

\begin{lemma}
As a real vector space, $\Mat_{2\times 2}(\CC)\cong \calh\oplus \calh i$.
\end{lemma}
\begin{proof}
Because of the differing signs of the determinants, $\calh\cap \calh i =\{\mathbf 0\}$, with $\mathbf 0$ the zero 
matrix, and the sum  $\calh + \calh i\subseteq \Mat_{2\times 2}(\CC)$  is direct.
As the matrices $\lambda(1), \lambda(\bfi), \lambda(\bfj), \lambda(\bfk)\in\calh$ are linearly independent over 
$\RR$, so are the matrices $\lambda(1)i, \lambda(\bfi)i, \lambda(\bfj)i, \lambda(\bfk)i\in\calh i$. This shows that 
$\dim_{\RR}( \calh\oplus \calh i) = 8 = \dim_{\RR}\Mat_{2\times 2}(\CC)$, whence the inclusion
$\calh\oplus \calh i\subseteq \Mat_{2\times 2}(\CC)$ is an equality.
\end{proof}

A different, but equivalent way to state the foregoing observation, is as follows, using that as real vector spaces $\HH\otimes_{\RR}\CC\cong \calh\oplus \calh i$ via 
$x\otimes (a+bi)\mapsto \lambda(xa) + \lambda(xb)i$: 
\begin{cor}
The $\CC$--algebra $\HH\otimes_{\RR}\CC$ is isomorphic to $\Mat_{2\times 2}(\CC)$ and, modulo this 
isomorphism, $\lambda$ is the canonical map 
$\lambda=\id_{\HH}\otimes \iota:\HH\cong \HH\otimes_{\RR}\RR\into \HH\otimes_{\RR}\CC$, 
where $\iota\colon \RR\into \CC$ is the inclusion of algebras. The map 
$\lambda(\ )i:\HH\xto{\ \cong\ }\HH\otimes_{\RR}\RR i\into \Mat_{2\times 2}(\CC)$ has as image $\calh i$.
\qed
\end{cor}
A further consequence is the following.
\begin{cor} \label{Cor:quaternionsunimodular}
The set $\lambda(\HH_{1})\cup \lambda(\HH_{1})i\subset \Mat_{2\times 2}(\CC)$ forms a group 
under multiplication, namely the unimodular unitary group
\begin{align*}
\Un(2)^{\pm 1} =\{A\in \GL(2,\CC)\mid A\cdot \overline A^{T} = \mathbf 1,  \ \det A = \pm 1\}\leqslant \Un(2)
\end{align*}
of unitary $2\times 2$ matrices of determinant $\pm 1$.\qed
\end{cor}

\section{The Clifford Algebra of Euclidean $3$--Space} \label{Sec:Clifford3}

Here we define the players in the upper part of the magic square: the (s)pin groups for Euclidean $3$-space. We use ad-hoc definitions to show Theorem \ref{thm:main}. In Section \ref{Sec:CliffordGeneral} general Clifford algebras and their (s)pin groups will be defined.

\begin{sit} \label{Pauli-matrices}
Let's have a closer look at the matrices $\lambda(1)i, \lambda(\bfi)i, \lambda(\bfj)i, \lambda(\bfk)i\in\calh i
\subset \Mat_{2\times 2}(\CC)$ from above. One has, obviously,
\begin{align*}
\lambda(1)i&=
\begin{pmatrix}
i&0\\0&i
\end{pmatrix}\,,\quad
\lambda(\bfi)i=
\begin{pmatrix}
-1&0\\0&1
\end{pmatrix}\,,\quad
\lambda(\bfj)i=
\begin{pmatrix}
0&-i\\i&0
\end{pmatrix}\,,\quad
\lambda(\bfk)i=
\begin{pmatrix}
0&1\\1&0
\end{pmatrix}\,.
\end{align*}
Up to relabeling and a sign, the last three of these are Pauli's famous {\em spin matrices\/} from 
quantum mechanics. More precisely, the classical {\em Pauli matrices\/} are
\begin{align*}
e_{1} = \lambda(\bfk)i = 
\begin{pmatrix}
0&1\\1&0
\end{pmatrix}\,\quad
e_{2} = \lambda(\bfj)i = 
\begin{pmatrix}
0&-i\\i&0
\end{pmatrix}\,,\quad
e_{3} =-\lambda(\bfi)i =
\begin{pmatrix}
1&0\\0&-1
\end{pmatrix}\,.
\end{align*}
\end{sit}

\begin{sit}
You check immediately that 
\[
e_{123}=e_{1}e_{2}e_{3} = \lambda(1)i =
\begin{pmatrix}
i&0\\0&i
\end{pmatrix}\,,
\]
and that the Pauli matrices satisfy 
\begin{align*}
e_{\nu}^{2} &=\mathbf{1}&&\text{for $\nu=1,2,3$,}\\
e_{\mu}e_{\nu} +e_{\nu}e_{\mu} &= \mathbf 0&&\text{for $\mu,\nu=1,2,3; \ \mu\neq \nu$.}
\end{align*}
\end{sit}

An immediate application of these relations is as follows.
\begin{cor}
For $x_{1}, x_{2}, x_{3}\in\RR$, one has in the matrix algebra $\Mat_{2\times 2}(\CC)$ that
\begin{align*}
(x_{1}e_{1}+ x_{2}e_{2} + x_{3}e_{3})^{2} = (x_{1}^{2} +x_{2}^{2}+ x_{3}^{2})\mathbf 1\,.
\end{align*}
Put differently, the real vector space spanned by the Pauli matrices inside $\Mat_{2\times 2}(\CC)$
and endowed with the quadratic form given by squaring inside the matrix algebra is Euclidean $3$--space.
\end{cor}

As for the remaining products of two Pauli matrices, writing $e_{\mu\nu}=e_{\mu}e_{\nu}$,
\begin{align*}
e_{12} = \lambda(\bfi)= 
\begin{pmatrix}
i&0\\0&-i
\end{pmatrix}\,,\quad
e_{13} = \lambda(\bfj)= 
\begin{pmatrix}
0&-1\\1&0
\end{pmatrix}\,,\quad
e_{32} = \lambda(\bfk)= 
\begin{pmatrix}
0&-i\\-i&0
\end{pmatrix}\,.
\end{align*}
Putting all this together re-establishes the following classical results. Note here that we will define Clifford algebras in general in Section \ref{Sec:CliffordGeneral}.
\begin{theorem} \label{Thm:Cl3}
\begin{enumerate}[\rm (1)]
\item 
The algebra $\Mat_{2\times 2}(\CC)$ is the {\em Clifford algebra 
$\Cl=\Cl_{3}=\Cl_{3,0}$ of Euclidean $3$--space}, 
$E=\RR e_{1}\oplus\RR e_{2}\oplus\RR e_{3}$ equipped with the standard Euclidean form
\[
Q(x_{1}e_{1}+ x_{2}e_{2} + x_{3}e_{3}) =x_{1}^{2} +x_{2}^{2}+ x_{3}^{2}\,.
\]
In particular, as a real algebra $\Mat_{2\times 2}(\CC)$ is generated by the three Pauli matrices.
\item
The matrices $\mathbf 1,\  e_{1}, e_{2},e_{3},\ e_{12},e_{13},e_{23},\ e_{123}$ form a real basis of this algebra.

\item
The {\em even\/} Clifford subalgebra is $\Cl^{\uzero}=\calh =\lambda(\HH) = 
\RR\mathbf 1\oplus\RR e_{12}\oplus\RR e_{13}\oplus\RR e_{23}$.

\item
The {\em odd\/} part of this Clifford algebra is $\Cl^{\uone}=\calh i =\lambda(\HH)i = 
\RR e_{1}\oplus\RR e_{2}\oplus\RR e_{3}\oplus\RR e_{123}$.
\item
Using the real vector space grading given by the product-length, one has the direct sum decomposition
into subspaces
\begin{align*}
\Cl^{0}&= \RR\mathbf 1\,,\\
\Cl^{1}&= \RR e_{1}\oplus\RR e_{2}\oplus\RR e_{3} =\mathfrak{su}(2){\cdot} i\,,\\
\Cl^{2}&= \RR e_{12}\oplus\RR e_{13}\oplus\RR e_{23} = \mathfrak{su}(2)\,,\\
\Cl^{3}&= \RR e_{123} =\RR\mathbf 1i\,.
\end{align*}

\end{enumerate}
\end{theorem}

\section{The Pin Groups for Euclidean $3$--Space} \label{Sec:Pin3}

The (multiplicative) group of {\em homogeneous\/} invertible elements in the $\ZZ/2\ZZ$--graded Clifford algebra 
$\Cl=\Cl_3$ is denoted $\Cl^{*}_{hom}$. In the case at hand, $\Cl^{*}_{hom} = \lambda(\HH^{*})\cup \lambda(\HH^{*})i$. The following three involutions can be defined more generally on the Clifford algebra of any quadratic space $(V,q)$, see \ref{CliffordInvolution} later. Here we give ad-hoc definitions, using the direct sum decomposition from Theorem \ref{Thm:Cl3} (5). For $x \in \Cl$ write $x=[x]_0 + [x]_1 + [x]_2 + [x]_3$, where $[-]_i$ is in $\Cl^{i}$. 

\begin{itemize}
\item The {\em principal automorphism\/} $\alpha$ of the Clifford algebra $\Cl$ 
of $x$ is given by $\alpha(x)=[x]_0 - [x]_1 + [x]_2 - [x]_3$.
\item The {\em transpose\/} $t$
  of the Clifford algebra $\Cl$
is given by $t(x)=[x]_0 + [x]_1 - [x]_2 - [x]_3$.
\item {\em Clifford conjugation\/} is the composition $\overline{x} =\alpha\circ t(x)=t\circ\alpha(x)$. It is an anti-automorphism of the algebra.
\end{itemize}
In our case, this amounts to the following $\RR$--linear maps, specified on the basis and on the homogeneous 
components of the product--length grading. 
\begin{align*}
\begin{array}{c|rrrrrrrr|rrrr}
&\mathbf 1&e_{1}& e_{2}& e_{3}& e_{12}& e_{13}& e_{23}& e_{123}&\Cl^{0}&\Cl^{1}&\Cl^{2}&\Cl^{3}\\
\hline
\alpha&1&-e_{1}& -e_{2}& -e_{3}& e_{12}& e_{13}& e_{23}& -e_{123}&+&-&+&-\\
t &1&e_{1}& e_{2}& e_{3}& -e_{12}& -e_{13}& -e_{23}& -e_{123}&+&+&-&-\\
\overline{(\ )} =\alpha t&1&-e_{1}& -e_{2}& -e_{3}& -e_{12}& -e_{13}& -e_{23}& e_{123}&+&-&-&+
\end{array}
\end{align*}
The following observation is now easily verified.
\begin{proposition}
The map $i_{-}:\Cl^{*}_{hom}\to \Aut_{\RR-alg}(\Cl)$ given by
\begin{align*}
i_{u}(x) = \alpha(u)xu^{-1}
\end{align*}
is a group homomorphism into the group of $\RR$--algebra automorphisms of the Clifford algebra.

\end{proposition}

\begin{sit} One uses these maps to define {\em norms\/} on the Clifford algebra,
\begin{align*}
N(x) &= \overline x x= \alpha(t(x))x\,,\\
N'(x) &= t(x)x\,.
\end{align*}
These norms define various group homomorphisms on the group of homogeneous invertible elements
in $\Cl$.
\end{sit}

\begin{lemma}
The restrictions of $N, N', N^{2}=(N')^{2}$ to $\Cl_{hom}^{*}$ yield group homomorphisms into 
$\RR^{*}{\cdot}\mathbf 1\cong \RR^{*}$.

These maps and their images can be identified on $\Mat_{2\times 2}(\CC)$ as follows.
\begin{align*}
\begin{array}{c|c|c|c|c}
&\lambda(\HH^{*})=(\Cl^{*})^{\uzero}&\text{image of $(\Cl^{*})^{\uzero}$}&\lambda(\HH^{*})i=(\Cl^{*})^{\uone}
&\text{image of $(\Cl^{*})^{\uone}$}\\
\hline
N&\det&\RR_{>0}&\det& \RR_{<0}\\
N'&\det& \RR_{>0}&-\det& \RR_{>0}\\
N^{2}=(N')^{2}&\det^{2}& \RR_{>0}&\det^{2}& \RR_{>0}
\end{array}
\end{align*}
\end{lemma}

\begin{proof}
We saw already earlier that for any quaternion $x$ and the corresponding matrix 
$\lambda(x)\in \lambda(\HH) =\Cl^{\uzero}$ one has $N(x) =\det\lambda(x)\mathbf 1$, 
with $\det\lambda(x) \geqslant 0$ . \\
If $y\in \Cl^{\uone}$, then $yi\in \Cl^{\uzero}$ and
$yi = ye_{123}\in\Mat_{2\times 2}(\CC)$. Thus, knowing $N$ on $ \lambda(\HH)$, we have first
$N(yi) =\det(yi)\mathbf 1 = -\det(y)\mathbf 1$, while by multiplicativity of the norm $N(yi)=N(y)N(e_{123})$
and $N(e_{123})= e_{123}^{2}= -\mathbf 1$, whence $N(y) =\det(y)\mathbf 1$ as claimed. 
The rest of the table follows.
\end{proof}
In the literature there are, confusingly enough, at least, three(!) competing definitions of the Pin groups 
and we record those now.
\begin{defn} 
With notation as introduced, set
\begin{align*}
\Pin^{+}(3) &= \ker(N\colon \Cl^{*}_{hom}\lto \RR^{*})\,,&&\text{the Pin group in \cite{ABS} or \cite{Knus}}\,,\\
\Pin'^{+}(3) &= \ker(N'\colon \Cl^{*}_{hom}\lto \RR^{*})\,,&&\text{the Pin group in \cite{Scharlau}}\,,\\
\PIN^{+}(3) &= \ker(N^{2}=(N')^{2}\colon \Cl^{*}_{hom}\lto \RR^{*})\,,&&\text{the Pin group in 
\cite{CBr, Porteous1, Lou}}\,.
\end{align*}
Clearly, $\Pin^{+}(3) , \Pin'^{+}(3)$ are subgroups of $\PIN^{+}(3)$, whence we call the latter the 
{\em big Pin group}.

Intersecting each of these Pin groups with $(\Cl^{*})^{\uzero}$ returns the same {\em spinor group\/} $\Spin^{+}(3)$.

As for notation, the superscript ``$+$'' refers to the positive definiteness of the quadratic form, the
argument $(3)$ reminds us that we are working with $3$--dimensional real space.
\end{defn}
The relation between the groups $\SU(2), \Un(2)^{\pm}$ and the (s)pin groups is now immediate.
\begin{theorem} \label{Thm:3Pindifferences}
One has
\begin{enumerate}[\rm(1)]
\item $\SU(2)=\Spin^{+}(3)=\Pin^{+}(3)$,
\item $\Un(2)^{\pm1}=\PIN^{+}(3) = \Pin'^{+}(3)$.
\end{enumerate}
\end{theorem}
\begin{proof}
We saw above that $\Cl^{*}_{hom}= \lambda(\HH^{*})\cup \lambda(\HH^{*})i$. Now the norm $N$ is positive
on $\lambda(\HH^{*})$, negative on $\lambda(\HH^{*})i$, thus, $\ker(N) = \lambda(\HH^{*})\cap \det^{-1}(1)
=\lambda(\HH_{1})$, by definition of $\HH_{1}$. We already identified $\HH_{1}$ with $\SU(2)$ before (see \ref{unitquaternions=SU(2)}),
thus, $\SU(2)=\Pin^{+}(3)$. Because $\Pin^{+}(3)$ is already contained in $(\Cl^{*})^{\uzero}$, it agrees with 
$\Spin^{+}(3)$.

Now $\Un(2)^{\pm1}=\SU(2)\cup\SU(2)i =\PIN^{+}(3)$, as $\ker N^{2} =N^{-1}(\{\pm 1\})$. Also, 
$\ker N^{2} = \ker (N')^{2} =\ker N'$, as $N'$ is everywhere positive. Finally, intersecting any of the three
Pin groups with $(\Cl^{*})^{\uzero}$ yields the same group $\Spin^{+}(3)=\SU(2)$.
\end{proof}

\begin{remark} \label{Rmk:3PinDefs}
Why are there three different kinds of Pin groups? The authors do not really know, but there are some subtle
differences in their uses as we will see below when considering this in the context of general Clifford algebras 
associated to quadratic forms over fields.

It seems that geometers and physicists prefer the big version, as that group surjects onto the orthogonal group,
see below, while algebraists or arithmetic geometers seem to prefer the ``small'' version, as it allows 
for introduction of the {\em spinor norm\/} on the orthogonal group, with the small pin group mapping 
onto its kernel. We do not study this spinor norm here further.

\end{remark}

Having clarified the relationship between $\SU(2)$ and $\Un(2)^{\pm 1}$ and the (s)pin groups, we now
turn to the surjectivity of $\rho$ and $\widetilde\rho$.

\section{Reflections} \label{Sec:Reflections}
Let $v = x_{1}e_{1}+ x_{2}e_{2} + x_{3}e_{3}$ be a nonzero real linear combination 
of the Pauli matrices, so that $v^{2}=  (x_{1}^{2} +x_{2}^{2}+ x_{3}^{2})\mathbf 1\neq 0 \in \Cl^{\uzero}$. It follows 
that $v$ is an invertible element in $(\Cl^{*})^{\uone}$, in particular, $\tfrac{1}{\sqrt{Q(v)}}v\in \Un(2)^{\pm 1}$,
and as a matrix in $\Mat_{2\times 2}(\CC)$ it represents a complex reflection.

\begin{lemma}
For $v$ as just specified, $\widetilde\rho(v)\in\Or(3)$ is a reflection.
\end{lemma}
\begin{proof}
  Note that $iv \in  \su(2)$ and $v^2=1$.   So we compute that
  $$\tilde \rho (v)(iv) = (\det v) v (iv) v^{-1}
   = -iv.$$
  Note that the inner product can be represented as
  $vw+wv = 2(v \cdot w)\mathbf 1$ so if we suppose that $w$ is a vector perpendiculat to $v$, this means that $vw+wv=0$.
    So we get that
  $$\tilde\rho(v)(iw)  = (\det v) v iw v^{-1}
   = (-1) (-iw)(v)(v^{-1}) =iw.$$
\end{proof}

\begin{corollary}
  The maps $\tilde{\rho}$ and $\rho$ are surjective.
\end{corollary}
\begin{proof}
  The above lemma shows that we can obtain any reflection in the image of $\tilde{\rho}$, since by the Cartan--Dieudonn\'e theorem, any element of $\Or(3)$ is a product of at most 3 reflections.  Now we consider the surjectivity of $\rho$.  Let $x \in \SO(3)$.  Then there are two reflections $a,b$ in $\Or(3)$ such that $x=ab$.  There are two elements $v,w \in \Un(2)$ with determinant $-1$ such that $ a =\tilde{\rho}(v), b=\tilde{\rho}(w)$.  So we see that $x=\tilde{\rho}(vw)$.
\end{proof}

Thus we have commented on all the maps in the square of Theorem \ref{thm:main} and proven all claims. \qed 

\section{General Clifford algebras} \label{Sec:CliffordGeneral}
\begin{sit}
In this part of the paper, we we will look at general Clifford algebras and their (s)pin groups and how they fit into square diagrams similar to \eqref{Diag:squareGroups}, see Theorem \ref{Thm:squaregeneral}. We end the paper with some examples of real Clifford algebras in small dimensions.

 First recall some relevant definitions and facts from \cite[\S 9]{ALG9} and \cite{ABS}, see
\cite{Scharlau1, Scharlau} for a nice exposition. When it comes to Clifford algebras of any signature over the real numbers, \cite{Lou} and \cite{CBr} are the most valuable references.
\end{sit}

\subsection*{Quadratic Spaces}
The following makes sense over any commutative ring $R$.
\begin{defn}
A {\em quadratic form\/} on an $R$--module $V$ over $R$ is a map $q:V\to R$ that satisfies
\begin{enumerate}[\rm (a)]
\item $q(rv)=r^{2}q(v)$ for all $r\in R, v\in V$, and
\item The {\em polarization\/} $b_{q}:V\times V\to R$ of $q$, given by
\begin{align*}
b_{q}(v,w) = q(v+w)-q(v)-q(w)\,,\quad \text{for $v,w\in V$,}
\end{align*}
is {\em $R$--bilinear}. 
As it is always symmetric, it is often also just called the {\em bilinear symmetric form\/} associated with $q$.
\end{enumerate}

The pair $(V,q)$ is then called
a {\em quadratic module} over $R$.

A morphism $\vp\colon (V,q)\to (V',q')$ is a $R$--linear map $\vp\colon V\to V'$ such that 
$q'(\vp(v))=q(v)$ for all $v\in V$. Such morphisms%
\footnote{In \cite{Scharlau} only {\em injective\/} such maps are allowed.}
are also called {\em orthogonal maps}.
With these morphisms, quadratic $R$--modules form a category $\mathsf{Quad}_{R}$. This category admits
arbitrary direct sums, $\oplus_{i\in I}(M_{i},q_{i}) = \left(\oplus_{i\in I}M_{i}, q(\sum_{i}m_{i}) = 
\sum_{i}q_{i}(m_{i})\right)$.

\end{defn}

\begin{remark}
If $(V,q)$ is a quadratic module, then for $v\in V$ one has $b_{q}(v,v)=2q(v)$. Thus, a symmetric bilinear form
recovers its underlying quadratic map $q$ as soon as $b_{q}(v,v)\in 2R$ for all $v\in V$ and $2$ is a 
non-zero-divisor in $R$. Both conditions are, of course, satisfied as soon as $2$ is invertible in $R$.
In that case, we call $(v,w) = \tfrac{1}{2}b_{q}(v,w)$ the scalar product defined by $q$. Note that then
$(v,v) = q(v)$, and the data given by the quadratic form or its associated scalar product are equivalent.

This is the key reason to restrict to the case that $2\in R$ is invertible.
\end{remark}

\begin{sit}
Of particular importance is the {\em orthogonal group\/} $\Or(q)$ of $(V,q)$, the automorphism group of $(V,q)$ in 
the category of quadratic modules. By the very definition, this is a subgroup of $\GL(V)=\Aut_{R}(V)$. 
In fact, the whole theory of quadratic modules and of Clifford algebras is nothing but an attempt to understand 
the orthogonal groups thoroughly.
\end{sit}

\begin{example}
If $R$ is a field and $V$ a finite dimensional vector space over it, then for any quadratic form $q$ on $V$
one can consider a basis $e_{1},...., e_{n}$ of $V$ and the matrix $B=(b_{q}(e_{i}, e_{j}))$. 
If further $\car R\neq 2$, then the quadratic form can be diagonalized by the Gram--Schmidt algorithm, in that 
there exists a basis $e_{1},...., e_{n}$ of $V$ such that the matrix $B=(b_{q}(e_{i}, e_{j}))$ is diagonal.

The quadratic form $q$ is {\em regular\/} (or {\em nonsingular\/}), if for some basis, and then for any, the 
matrix $B$ is invertible. If the quadratic form is diagonalized, then this means simply that $b_{q}(e_{i}, e_{i})
= 2q(e_{i}) \neq 0$ for each $i$.  
\end{example}

\begin{example}
The key examples are given by the regular quadratic forms on the real vector space 
$V=\bigoplus_{i=1}^{n} \RR e_{i}$ over $R=\RR$ for some positive integer $n$.
 
By Sylvester's Inertia Theorem, one can assume that, up to base change, 
the quadratic form is one of
\begin{align*}
q_{m,n-m}(x_{1}e_{1}+\cdots + x_{n}e_{n}) = x_{1}^{2} +\cdots + x_{m}^{2} -(x_{m+1}^{2}+\cdots +x_{n}^{2})
\end{align*}
where $m=0,...,n$. The form $q_{n} = q_{n,0}$ is the usual positive definite Euclidean norm, while
$q_{0,n} = -q_{n,0}$ is the negative definite form. The other forms are indefinite in that there are nonzero
{\em isotropic\/} vectors, elements $v\in V$ such that $q(v)=0$. The isotropic vectors form the {\em light cone\/}
$\{v\in V\mid q(v)=0\}\subset V$, a quadratic hypersurface.

One writes $\RR^{m,n-m}$ for the quadratic space $(\RR^{n}, q_{m,n-m})$.
\end{example}

\begin{remark} \label{Rmk:reflections}
If $r\in R$ is some element from the base ring $R$, then with $(V,q)$ also $(V,rq)$ is a quadratic $R$--module
in that $(rq)(v) = r(q(v))$ is again $R$--quadratic. Further, if $\vp\in \Or(q)$, then $\vp$ is, trivially, in $\Or(rq)$,
thus $\Or(q)\subseteq \Or(rq)$. If $r\in R^{*}$, the group of invertible elements of $R$, then $\Or(q) = \Or(rq)$,
so that $q$ and $rq$ define the {\em same\/} orthogonal group.

This comment should be kept in mind in what follows.
\end{remark}

\subsection*{Reflections, Take two}
We next exhibit special elements in $\Or(q)$, those given by {\em reflections}.
\begin{sit}
\label{reflect}
Elements $v,w$ in $V$ are {\em orthogonal\/} with respect to $q$ if $b_{q}(v,w)=0$. If for some $v\in V$ one has 
$q(v)\in R^{*}$, the (multiplicative) group of invertible elements in $R$, then the map 
\begin{align*}
s_{v}\colon V\to V\,,\quad s_{v}(w) = w - b_{q}(v,w)q(v)^{-1}v
\end{align*}
is defined, $R$--linear, and satisfies $s_{v}(w)=w$ if $w$ is orthogonal to $v$, while $s_{v}(v) =-v$, 
because $(v,v)= 2q(v)$.
Moreover, $s_{v}$ is an orthogonal endomorphism of $V$, as
\begin{align*}
q(s_{v}(w)) &= q(w) + q(-b_q(v,w)q(v)^{-1}v) +(w,-b_q(v,w)q(v)^{-1}v)\\
&=q(w) + b_q(v,w)^{2}q(v)^{-2}q(v) - b_q(v,w)q(v)^{-1}b_q(w,v)\\
&=q(w)\,,
\end{align*}
using the definition of the polarization in the first equality, the quadratic nature of $q$ and bilinearity of the 
polarization in the second, and symmetry in the last. Further, $s_{v}^{2}=\id_{V}$, as
\begin{align*}
s_{v}^{2}(w) &= s_{v}(w - b_q(v,w)q(v)^{-1}v)\\
&=w - b_q(v,w)q(v)^{-1}v - b_q(v,w - b_q(v,w)q(v)^{-1}v)q(v)^{-1}v\\
&=w - b_q(v,w)q(v)^{-1}v  -b_q(v,w)q(v)^{-1}v + b_q(v,w)q(v)^{-1}b_q(v, v)q(v)^{-1}v\\
&=w - 2b_q(v,w)q(v)^{-1}v + 2b_q(v,w)q(v)^{-1}v\\
&=w\,,
\end{align*}
where the first and second equality just use the definition of $s_{v}$, the third uses linearity of $(v,\ )$, the fourth
invokes again that $(v,v) = 2q(v)$, and the last one is obvious.

\begin{example}
In the most important case, when $R$ is a field, thus, $V$ a vector space, then $q(v)\in R^{*}$ just means
$q(v)\neq 0$. If also $(v,v)=2q(v)\neq 0$, that is, $R$ is not of characteristic $2$, then the set $v^{\perp}$ of 
elements orthogonal to $v$ forms a hyperplane in $V$, kernel of the $R$--linear form $(v,\ ):V\to R$, and $s_{v}$ is 
the usual reflection in that hyperplane.
\end{example}

Because of this example, $s_{v}$ is generally called the {\em reflection in the hyperplane perpendicular to $v$},
regardless of the nature of the quadratic form.
\end{sit}

\begin{sit}
Note that $s_{rv} = s_{v}$ for any $r\in R^{*}$, as $q$ is a quadratic form. 
More generally, if $v, v'\in V^{*}$, when is $s_{v}=s_{v'}$? This requires that for any $w\in V$ we have
$s_{v}(w) = s_{v'}(w)$, equivalently, 
\begin{align*}
(v,w)q(v)^{-1}v = (v',w)q(v')^{-1}v'\,.
\end{align*}
Taking $w=v$, this specializes to
\begin{align*}
2v&= (v',v)q(v')^{-1}v'\,.
\end{align*}
Taking instead $w=v'$, we get
\begin{align*}
(v,v')q(v)^{-1}v = 2 v'\,.
\end{align*}
Thus, if $2\in R^{*}$, then $s_{v}= s_{v'}$ if, and only if, $v\equiv v'\bmod R^{*}$.
\end{sit}

In light of this simple fact, we make the following definition.

\begin{defn}
Let $(V,q)$ be a quadratic $R$--module. With
\begin{align*}
V^{*}(q) &= \{v\in V\mid q(v)\in R^{*}\}\subseteq V
\end{align*}
we have a canonical map $s_{-}: V^{*}(q)\to \Or(q)$ that sends $v\in V^{*}(q)$ to $s_{v}\in \Or(q)$.
The map $s_{-}$ factors through the orbit space $R^{*}\backslash V^{*}$ and the resulting map,
still denoted $s_{-}:R^{*}\backslash V^{*}\to \Or(q)$ is injective if $2$ is invertible in $R$.
We call the set $\Refl(q) = \Imm(s_{-})\subseteq \Or(q)$ the set of reflections in $\Or(q)$.

If $r\in R^{*}$, then $V^{*}(q)= V^{*}(rq)$ and $\Or(q)=O(rq)$ and the map $s_{-}$ depends only on the class of 
$q\bmod R^{*}$. In particular, $\Refl(q) =\Refl(rq)$, cf.~Remark \ref{Rmk:reflections}. 
\end{defn}

\subsection*{Clifford Algebras}
\begin{defn}
Let $q:V\to R$ be a quadratic form.
The {\em Clifford algebra\/}  $\Cl(q)$ of $q$ is the quotient of the tensor algebra $T_{R}V$ 
modulo the two-sided ideal $I(q)$ generated by the elements of the form $v\otimes v-q(v)\cdot 1$ for $v\in V$.

The canonical composition $V=T^{1}_{R}V\into T_{R}V\onto \Cl(q)$ is denoted $i_{q}$, but we often
write $v\in\Cl(q)$ instead of $i_{q}(v)$. This is justified when $V$ is a projective $R$--module, as then $i_{q}$ 
is injective --- and also $R\cong R\cdot 1\subseteq \Cl(q)$; see \cite[IV, Cor. 1.5.2]{Knus} and 
\cite[\S 9 Thm 1]{ALG9} .
\end{defn}

\begin{remark}
\label{rem2.6}
For $v,w\in V$, one has $(v+w)^{\otimes 2} = v\otimes v+ v\otimes w + w\otimes v +w\otimes w$ in $T_{R}V$.
Passing to $\Cl(q)$ this equality becomes
\begin{align*}
q(v+w) &= q(v) + i_{q}(v)i_{q}(w) +  i_{q}(w)i_{q}(v) + q(w)\,,
\intertext{and dropping $i_{q}$ from the notation, this rearranges to}
(v,w)1&= vw + wv \in R\cdot 1\subseteq \Cl(q)\,.
\end{align*}

\end{remark}

\begin{sit}
Because the generators of $I(q)$ involve only terms of even tensor degree, the natural grading by tensor degree
on $T_{R}V$ descends to a $\ZZ/2\ZZ$--grading on the Clifford algebra, 
$\Cl(q)=\oplus_{\underline i\in\ZZ/2\ZZ}\Cl(q)^{\underline i}$. The summand $\Cl(q)^{\uzero}$ is the {\em even\/}
Clifford algebra defined by $q$, while the summand $\Cl(q)^{\uone}$ of {\em odd\/} elements is a bi-module over it. By definition, the elements in $i_{q}(V)$ are of odd degree. 

The subset of homogeneous elements in $\Cl(q)$ with respect to this $\ZZ/2\ZZ$--grading is denoted
$\Cl(q)_{hom}$ and $|\ |\colon \Cl(q)_{hom}\to \ZZ/2\ZZ$ is the map recording
the {\em degree\/} (or {\em parity\/}) of a homogeneous element. 
\end{sit}

\begin{sit}
Inside $\Cl(q)_{hom}$ we have the {\em graded centre\/} of $\Cl(q)$, defined as
\[
Z =\left \{u\in \Cl(q)_{hom}\mid \forall x\in \Cl(q)_{hom}: ux = (-1)^{|u||x|}xu \right\}\,.
\]
Note that this graded centre is generally not a ring as the sum of two homogeneous elements of different degrees is 
no longer homogeneous. However, with $Z^{\underline i} = Z\cap\Cl(q)^{\underline i}$ one has
\begin{itemize}
\item $Z=Z^{\uzero}\cup Z^{\uone}$,
\item $Z^{\uzero}\oplus Z^{\uone}$ is a graded--commutative $R$--subalgebra of $\Cl(q)$,
\item $Z^{\uzero}$ is a commutative $R$--algebra and constitutes the intersection of the
ordinary, ungraded centre of $\Cl(q)$ with the even Clifford algebra $\Cl(q)^{\uzero}$.
\end{itemize}

For later use, we also record the following: 
\begin{lemma}\label{lemma:gradedCentre}
  Let $(V,q)$ be a quadratic module over a field $R$ of characteristic not 2.  Then the graded centre of $\Cl(q)$ is $Z \cong Z^{\underline{0}} \cong R$.
  \end{lemma}
  
  \begin{proof}
  See e.g.~\cite[5,Thm.2.1, p.~109]{LamQuadratic}.
  \end{proof}
  \end{sit}

\begin{sit}
  Let $A$ be an associative $R$ algebra.  A {\em Kummer subspace} $W$ of $A$ is an $R$-submodule of $A$ such that $w^2 \in R$ for all $w \in W$. 
  Let $\mathsf{AlgKgr_R}$ be the category of pairs $(A,W)$ of $\ZZ/2\ZZ$ graded associative $R$-algebras $A$ with a Kummer subspace $W \subseteq A^{\underline{1}}$.
  Morphisms are given by $\phi:(A,W) \to (B,U)$ where $\phi:A \to B$ is a graded $R$-algebra homomorphism and $\phi(W) \subseteq U$.
The assignment $(A,W) \mapsto (W,(\ )^{2})$ from $\mathsf{AlgKgr}_R$ to quadratic $R$--modules defines a forgetful functor
$
F: \mathsf{AlgKgr}_{R}\lto \mathsf{Quad}_{R}
$
and this functor 
admits a left adjoint,
\begin{align*}
\Cl:\mathsf{Quad}_{R}\lto \mathsf{AlgKgr}_{R}
\end{align*}
that sends the quadratic $R$--module $(M,q)$ to the pair $(\Cl(q),M)$ of the Clifford algebra and the Kummer subspace $M\subseteq \Cl(q)$.
Note that the adjunction $\mathrm{id} \xrightarrow{\simeq} F \circ \Cl$ is a natural isomorphism so $\Cl$ is a fully faithful functor to $\mathsf{AlgKgr}_R$.

 \end{sit}

\begin{sit} \label{CliffordInvolution}
Let $\Aut_{R-alg}^{\uzero}(\Cl(q))$ denote the group of degree-preserving $R$--algebra automorphisms of 
$\Cl(q)$.

Each Clifford algebra comes equipped with special degree-preserving $R$--linear endomorphisms:
\begin{itemize}
\item[$\Or(q)$] The $R$--linear automorphisms $\vp:V\to V$ that respect the quadratic form in that
$q(\vp(v)) = q(v)$ for $v\in V$ form the orthogonal group $\Or(q)$. These are thus the automorphisms of $(V,q)$
within the category of quadratic $R$--modules.

Each such automorphism of $V$ extends uniquely to a degree preserving $R$--algebra automorphism of 
$\Cl(q)$, thus, defines a group homomorphism $\Or(q)\to \Aut_{R-alg}^{\uzero}(\Cl(q))$. 

For a quadratic $R$--module $V$ with $i_{q}$ injective, 
this group homomorphism is injective too, as its image consists precisely of those 
$\Phi\in \Aut_{R-alg}^{\uzero}(\Cl(q))$ that preserve $i_{q}(V)\cong V$, since $\Phi$ is determined uniquely by its action on $V$.
\item[($\alpha$)] 
In particular, multiplication by $-1$ on $V$ is in $\Or(q)$ and the corresponding unique algebra 
automorphism $\alpha:\Cl(q)\to \Cl(q)$ is called the {\em principal automorphism\/} of $\Cl(q)$ in \cite[\S 9]{ALG9},
or the {\em grade involution\/} in \cite{Lou}. 
This automorphism is the identity on $\Cl(q)^{\uzero}$, and acts through multiplication by $-1$ on $\Cl(q)^{\uone}$.
Thus, if $2\neq 0$ in $R$, then the $\ZZ/2\ZZ$--decomposition of $\Cl(q)$ is given by the $\pm 1$ eigenspaces of 
$\alpha$. Clearly, $\alpha^{2}=\id_{\Cl(q)}$.

Physicists call $\alpha$ the {\em (space time) parity operator}, see \cite[p.18]{CBr}.

\item[($t$)] The unique algebra antiautomorphism $t:\Cl(q)\to \Cl(q)$ that is the identity on $V\subseteq \Cl_{q}$.
It is denoted $\beta$ and called the {\em principal antiautomorphism\/} in \cite[\S 9]{ALG9}, and called \emph{reversion} in \cite{Lou},  while in \cite{ABS}
it is called the {\em transpose\/} and is denoted $t$ as here.

This antiautomorphism satisfies $t^{2}=\id_{\Cl(q)}$, thus, is an (anti-)involution.
\item[($t\alpha$)]
The composition $t\circ \alpha=\alpha\circ t:\Cl(q)\to \Cl(q)$ is called {\em conjugation} and denoted
$\overline x = t(a(x))$ for $x\in \Cl(q)$. It is an involutive antiautomorphism --- the unique antiautomorphism
that restricts to multiplication by $-1$ on $i_{q}(V)$.
\end{itemize}
\end{sit}

\begin{sit}
In analogy with the quaternions, the Clifford (square) {\em norm\/} of $x\in \Cl(q)$ is%
\footnote{Here we follow \cite[p. 228]{Knus}, which is opposite to the definition in \cite{ABS}, where
$N(x)=x\overline x$ is used.} 
$N(x) = \overline x x\in\Cl(q)^{\uzero}$.
In contrast to the quaternions though, usually $N(x)\neq x\overline x$ when $x\not\in i_{q}(V)$.
For $v\in V$, however, $\overline v = -v$ and one still has 
$v\overline v = \overline v v= -v^{2} =-q(v)1\in R\cdot 1\subseteq \Cl(q)^{\uzero}$.
Further, $\overline{N(x)} =N(x)$ for any $x\in \Cl(q)$, as 
$t\alpha(t\alpha(x)\cdot x) =t\alpha(x)\cdot  t\alpha(t\alpha(x)) = t\alpha(x)\cdot x$.
\end{sit}

\begin{remark} \label{Rmk:N=N'}
Once again, this definition is not uniquely used in the literature. For example, Bourbaki uses instead
$N'(x) = t(x)x\in \Cl(q)^{\uzero}$ and so does Scharlau in \cite[p. 335]{Scharlau}. 
The argument for this usage is that then $N'$ defines a quadratic form on $\Cl(q)$ that extends 
$q$ rather than $-q$ on $i_{q}(V)$.
Note that $N=N'$ on the even Clifford algebra $\Cl(q)^{\uzero}$, while $N'=-N$ on the odd part.

On the other hand, \cite {Lou} or \cite[p. 334]{Scharlau} use all of $\Cl(q)^{*}$ as the domain of $i_{u}$, 
which then is defined as $i_{u}(x) =\alpha(u)xu^{-1}$ in \cite{Scharlau}.

Our usage follows \cite{Knus}.
\end{remark}

\begin{proposition}
Denote $\Cl(q)^{*}_{hom}$  the multiplicative group of homogeneous invertible elements in $\Cl(q)$. The map
\begin{align*}
i\colon \Cl(q)^{*}_{hom}&\lto \Aut_{R-alg}^{\uzero}(\Cl(q))
\end{align*}
that sends $u\in \Cl(q)^{*}_{hom}$ to the graded inner automorphism $i_{u}$ with 
$i_{u}(x) = (-1)^{|u||x|}uxu^{-1}$ for $x\in \Cl(q)_{hom}$ is a group homomorphism
with kernel $Z\cap  \Cl(q)^{*}_{hom}$ .
Moreover, if $(V,q)$ is regular and $R$ is a field of characteristic not 2, then we have an exact sequence
$$ 1 \lto R^* \lto  \Cl(q)^{*}_{hom} \xrightarrow{\phantom{aa}i\phantom{aa}} \Aut_{R-alg}^{\uzero}(\Cl(q)) \lto 1.$$
\end{proposition}
\begin{proof}
  The kernel of $i$ is given by the graded centre which is described in Lemma~\ref{lemma:gradedCentre}.  The fact that $i$ is surjective follows from the fact that the Clifford algebra is a graded central simple $R$-algebra when $(V,q)$ is regular, and the graded Skolem-Noether Theorem in Section 5 of \cite{Vela}.
\end{proof}

\begin{remark}
Prior to \cite{ABS}, such as in \cite{ALG9}, instead of assigning the graded inner automorphism to $u$,
one assigned to $u\in \Cl(q)^{*}$ the inner automorphism $j_{u}(x) = uxu^{-1}$. 

This has the unpleasant side effect that $j$ does not always surject onto $\Or(q)$ as noted in
\cite{ALG9} and \cite{CBr}.

\end{remark}

\begin{defn}
Let $(V,q)$ be a quadratic $R$--module with $V$ projective over $R$ 
so that $\Or(q)\leqslant \Aut_{R-alg}^{\uzero}(\Cl(q))$ is a subgroup.
The {\em Clifford group\/} of $q$ is then $\Gamma(q) = i^{-1}(O(q))\leqslant  \Cl(q)^{*}_{hom}$ and the {\em special 
Clifford group\/} is $S\Gamma(q) = \Gamma(q)\cap \Cl(q)^{\uzero}$.

In other words, $\Gamma(q)$ consists of those homogeneous invertible elements $u$ from $\Cl(q)$ whose 
associated graded inner automorphism maps $V\subset \Cl(q)$ to itself.
\end{defn}

\begin{remark}
Being an algebra automorphism of $\Cl(q)$, the map $i_{u}$ is uniquely determined by its action on $i_{q}(V)$.
There, it is given by $i_{u}(v) = (-1)^{|u|}uvu^{-1}$, as $|v|=\uone$. Thus, $u\in \Cl(q)^{*}_{hom}$ is in 
$\Gamma(q)$ if, and only if, $uvu^{-1}$ is again in $i_{q}(V)$ for any $v\in V$, as multiplication by 
the sign $(-1)^{|u|}$ preserves $i_{q}(V)$. This is the characterization of $\Gamma(q)$ that one 
often finds in the literature.
\end{remark}

\begin{example}[cf. {\cite[IV (6.1.2)]{Knus}}]
The following classical observation is crucial for our study of reflection groups.

An element $v\in V$ is homogeneous of degree $|v|=\uone$. It is in $\Cl(q)^{*}_{hom}$ if, and only if, 
$q(v)\in R^{*}$, as $v^{2} = q(v)1$. Therefore, the set
\begin{align*}
V^{*} &= \{v\in V\mid q(v)\in R^{*}\}
\end{align*}
is a subset of $\Cl(q)^{*}_{hom}$. In fact, it is a subset of $\Gamma(q)$ in that $i_{v}|_{i_{q}(V)}$ is the orthogonal
automorphism that is the reflection in the mirror perpendicular to $v$. Indeed, for any $w\in V$,
\begin{align*}
i_{v}(w) &= -vwv^{-1}&&\text{by definition of $i_{v}$,}\\
&=-vwq(v)^{-1}v&&\text{as $v^{-1}= q(v)^{-1}v$,}\\
&=(wv - 2(v,w))q(v)^{-1}v&&\text{by Remark \ref{rem2.6},}\\
&= w - 2(v,w)q(v)^{-1}v&&\text{as $v^{2}= q(v)$,}\\
&= s_{v}(w)&&\text{as in \ref{reflect}.}
\end{align*}
\end{example}

\begin{proposition}
Let $(V,q)$ be a quadratic $R$--module with $V$ projective over $R$. 
If $u\in \Gamma(q)$, then $N(u)\in (Z^{\uzero})^{*}$.
\end{proposition}

\begin{proof}
For any $u\in \Cl(q)$, its norm $N(u)$ is in $\Cl(q)^{\uzero}$ as noted above. For any $v\in V$ and 
$u\in \Gamma(q)$ one has $i_{u}(v) = -\overline{i_{u}(v)}$, as $i_{u}(v)\in V$, and then, by definition of 
$i_{u}$, that 
$ -\overline{i_{u}(v)} = -\overline{(-1)^{|u|}uvu^{-1}} = (-1)^{|\overline{u^{-1}}|}\overline{u^{-1}}v\overline u = 
i_{\overline u^{-1}}(v)$, as $|v| =\uone, |u|=|\overline{u^{-1}}|$. Thus, $i_{u}(v) = i_{\overline u^{-1}}(v)$.
Therefore, $i_{\overline u}i_{u}= i_{N(u)}=\id_{V}$, but this means that $N(u)$ commutes with all $v\in V$, as
$N(u)$ is an even element. 
Because $N(u)$ is an invertible element along with $u$ and because $V$ generates $\Cl(q)$, the claim follows.
\end{proof}

\begin{cor}
Let $(V,q)$ be a quadratic $R$--module with $V$ projective over $R$. The norm map defines a group homomorphism
\begin{align*}
N\colon \Gamma(q)\lto (Z^{\uzero})^{*}
\end{align*}
from the Clifford group to the multiplicative group of units in the ring $Z^{\uzero}$. Recall that $(Z^{\uzero})^{*}$ is equal to $R^*$ if $V$ is regular and $R$ is a field of characteristic $\neq 2$.
\end{cor}

\begin{proof}
Using that conjugation is an antiautomorphism, $N(uu') = \overline{uu'}uu' = \overline{u'}\overline{u}uu'$.
Now $N(u)=\overline{u}u\in Z^{\uzero}$ by the preceding Proposition, thus, 
$\overline{u'}\overline{u}uu' = N(u)\overline{u'}u'=N(u)N(u')$.
\end{proof}

\begin{remark} \label{Rmk:Spin-equal}
Note that if $x \in \Cl^{\underline{0}}$, then $N(x)=1$ if and only if $N'(x)=1$, because $\alpha$ is the identity on the even part of the Clifford algebra, also cf.~Remark \ref{Rmk:N=N'}.
\end{remark}

Now we can finally define the $Pin$ and $Spin$ groups of a quadratic module. 
There are at least three definitions of the Pin group in use in the literature, cf.~Remark \ref{Rmk:3PinDefs}. For more detailed comments also see \cite{Ragnar-overflow}.

\begin{defn}
Let $(V,q)$ be a quadratic $R$--module with $V$ projective over $R$. The kernel of the group homomorphism
$N\colon \Gamma(q)\lto (Z^{\uzero})^{*}$ is the {\em Pin group\/} $\Pin(q)$, with
$\Spin(q) = \Pin(q)\cap S\Gamma(q)$ the {\em Spin group\/} of $q$. \\
The kernel of the group homomorphism
$N' \colon \Gamma(q)\lto (Z^{\uzero})^{*}$ is equal to the kernel of $N\colon \Gamma(q)\lto (Z^{\uzero})^{*}$ and so $\Pin'(q) \cap S\Gamma(q)$ yields the same Spin group. \\
The kernel of the group homomorphism $N^2=(N')^2 \colon \Gamma(q)\lto (Z^{\uzero})^{*}$ is the {\em big Pin group\/} $\PIN(q)$, with the {\em big Spin group}
$\mathsf{SPIN}(q) = \PIN(q)\cap S\Gamma(q)$.
\end{defn}

\begin{remark}
By Remark \ref{Rmk:Spin-equal} we only get two (potentially) different Spin groups. However, if $(V,q)$ regular and  $R^*/(R^*)^2 =  \{ \pm 1\}$, then $\mathsf{SPIN}(q)=\Spin(q)$. 
\end{remark}

By their very definition, $\Pin(q)$, and so also $\Spin(q)$ act through orthogonal $R$--linear maps on $V$.
Moreover, one can easily identify the kernels of these actions.
\begin{proposition}  \label{Prop:PinExact}
One has short exact sequences of groups
\begin{align*}
\xymatrix{
1\ar[r]&Z\cap \Cl(q)^{*}_{hom}\ar[r]&\Gamma(q)\ar[r]&\Or(q)\\
1\ar[r]&K\ar[r]&\Pin(q)\ar[r]&\Or(q)\\
1\ar[r]&L\ar[r]&\PIN(q)\ar[r]&\Or(q)\\
1\ar[r]&SK\ar[r]&\Spin(q)\ar[r]&\Or(q)
}
\end{align*}
where $K= \{z\in Z\cap \Cl(q)^{*}_{hom}\mid N(z)=1\}$, $L=\{ z \in Z \cap \Cl(q)^{*}_{hom}\mid N^2(z)=1\}$ and $SK= \{z\in (Z^{\uzero})^{*}\mid N(z)=1\}$.\\
 If $R^*/(R^*)^2 \subseteq \{ \pm 1 \}$, and $2$ is invertible in $R$, then the map $\PIN(q) \xrightarrow{} \Or(q)$ is surjective.
\end{proposition}
\begin{proof}
  The description of the kernels follow immediately from the definitions.
  To show the last statement, consider the reflection $s_v \in \Or(q)$ for the vector $v \in V$ with $q(v) \neq 0$.  We can replace $v$ with $\frac{v}{\sqrt{q(v)}}$ so now $q(v) = \pm 1$.  Now we have $N^2(v) = 1$ and so $v \in \PIN(q)$.  By the Cartan-Dieudonn\'e Theorem, we obtain the result.
  \end{proof}

\begin{defn}
  Let $(V,q)$ be a quadratic module over $R$ with $\rank V =n$.  We define the special orthogonal group 
  $$\SO(q) := \SL(R,n) \cap \Or(q) = \{ M \in \Or(q) \mid \det M =1\}.$$
\end{defn}
\begin{theorem} \label{Thm:squaregeneral}Let $(V,q)$ be a quadratic module over a field $R$ of characteristic $\neq 2$.
Let $P(q)$ be any of the three Pin groups.
  Then there are commutative diagrams
\begin{align*}
\xymatrix{
P(q)\ar[r]&\Or(q)\\
\Spin(q)\ar[r]\ar@{^{(}->}[u]&\SO(q)\ar[u]
}
 & \hspace{3cm}
 \xymatrix{
\PIN(q)\ar@{->>}[r]&\Or(q)\\
\mathsf{SPIN}(q)\ar[r] \ar@{^{(}->}[u]&\SO(q)\ar[u]
}
\end{align*}
\end{theorem}
\begin{proof}
  By definition we know that $\Spin(q) \leqslant P(q)$ for all choices of Pin group.  We also know that $P(q) \subseteq \Gamma(q)$  maps to  $\Or(q)$.
  Again by definition $\SO(q) \leqslant \Or(q).$  We need to show that if $x \in \Spin(q)$ then $i_x \in \Or(q)$ has $\det i_x =1$.  By definition $x$ is in $\Cl^{\underline{0}}$.
  Since $i_x \in \Or(q)$, the Cartan-Dieudonn\'e Theorem says that it is a product of $k$ reflections with $k \leq n$.
  So there are elements $v_1,\ldots,v_k \in V^*$ such that
  $$i_x =  \prod_{i=1}^k s_{v_i}.$$
  Since $\det s_{v_i} = -1$ for all $i$, we see that $\det i_x = (-1)^k$. So we get 
  $$i_x = \prod s_{v_i} = \prod i_{v_i} = i_{\prod v_i}.$$
  Note that the inverse of $i_{\prod v_i}$ is given by carrying out the reflections in the opposite order, that is, $(i_{\prod v_i})^{-1}=i_{\prod_{i=k}^1 v_i}$ and 
 we see that $$x \prod_{i=k}^1 v_i \in \ker i.$$
  Note that $\ker i $ is contained in the graded centre of $\Cl(q)$ which has degree zero by Lemma~\ref{lemma:gradedCentre}.
  Therefore $\prod_{i=1}^k v_i$ has degree zero and we are done. 
  \end{proof}

\begin{sit} {\bf Questions about (s)pin groups.}
 It would be interesting to determine the cokernels in the exact sequences in Prop.~\ref{Prop:PinExact} and in Thm.~\ref{Thm:squaregeneral}. Moreover, it is not clear in general in which cases the various (s)pin groups are the same or differ from each other. In particular, it would be interesting to see in which way $\Spin$ and the big $\mathsf{SPIN}$ differ when the ring $R$ does not satisfy $R^*/(R^*)^2 = \{ \pm 1 \}$.
 \end{sit}

Finally, to see clearly the confusion existing in the literature, consider the simplest non-trivial cases for real Clifford algebras.  

\subsection*{Real Clifford Algebras}
We write $\Cl_{m,n-m} = \Cl(q_{m,n-m})$. To save space, for a ring $A$ we set $^{2}A =A\times A$ with componentwise operations, $A(d) =\Mat_{d\times d}(A)$, the matrix ring over $A$. We will write $P^+(n)$ and $P^-(n)$ for the any of the Pin groups $P$ of $\Cl_{n,0}$ and $\Cl_{0,n}$, respectively, and use analogous notation for the Spin groups. Note that over $\RR$, the Spin groups $\mathsf{SPIN}$ and $\Spin$ coincide.

One has the following structure theorem for the algebras $\Cl_{m,n-m}$. 
\begin{theorem}[see {\cite[p. 217]{Lou}} for this table] \label{Thm:realClifftable}
The Clifford algebras $\Cl_{m,n-m}$ for $0\leqslant n <8$ are isomorphic to
{\footnotesize
\[
\begin{array}{c|ccccccccccccccc}
\hline 2m-n&-7&-6&-5&-4&-3&-2&-1&0&1&2&3&4&5&6&7\\
\hline n=0&&&&&&&&\RR\\
1&&&&&&&\CC&&^{2}\RR\\
2&&&&&&\HH&&\RR(2)&&\RR(2)\\
3&&&&&^{2}\HH&&\CC(2)&&^{2}\RR(2)&&\CC(2)\\
4&&&&\HH(2)&&\HH(2)&&\RR(4)&&\RR(4)&&\HH(2)\\
5&&&\CC(4)&&^{2}\HH(2)&&\CC(4)&&^{2}\RR(4)&&\CC(4)&&^{2}\HH(2)\\
6&&\RR(8)&&\HH(4)&&\HH(4)&&\RR(8)&&\RR(8)&&\HH(4)&&\HH(4)\\
7&^{2}\RR(8)&&\CC(8)&&^{2}\HH(4)&&\CC(8)&& ^{2}\RR(8)&&\CC(8)&& ^{2}\HH(4)&&\CC(8)\\
\hline
\end{array}
\]}
and $\Cl_{m,n-m+8}\cong \Cl_{m,n-m}(16)$.\qed
\end{theorem}

\begin{example}[$\mathbf{\Cl_{0,1}}$]
Here we consider $V=\RR e$ with $q(re) = -r^{2}$, thus, $\Cl_{0,1}\cong\CC$
by sending $e\mapsto i$, as in the table above. The even Clifford algebra is $\RR\cdot 1\subset \CC$, 
the real part, the odd part is $\RR i\subset \CC$, the purely imaginary numbers. 

The principal automorphism $\alpha$ corresponds to complex conjugation, the principal antiautomorphism $t$ 
is the identity on $\CC$.

Accordingly, we have 
\[
\begin{array}{|c|c|c|c|c|c|c|c|c|c|}
\hline
\Cl^{\uzero}&\Cl^{\uone}&\Cl^{*}&\Cl^{*}_{hom}&\Gamma&S\Gamma&O& N(a+b i)&N'(a+bi)\\
\hline
\vphantom{\frac{1^{2}}{2}}\RR{\cdot} 1&\RR{\cdot} i&\CC^{*}&\RR^{*}{\cdot} 1\cup\RR^{*}{\cdot} i& 
\Cl^{*}_{hom}&\RR^{*}{\cdot} 1& \{\pm1\}& a^{2} + b^{2}&(a+bi)^{2}\\
\hline
\end{array}
\]
Note that $\Gamma = (\Cl_{0,1})^{*}_{hom}$ as $\Cl_{0,1}$ is commutative, thus, $(-1)^{|u||x|}uxu^{-1} =
(-1)^{|u||x|}x$ and preservation of $i_{q}V$ is automatic.
 For use below, we record as well that $(Z^{\uzero})^{*}=\RR^{*}{\cdot} 1$ so that
\begin{align*}
\Pin^{-}(1) & =  \ker(N:\Gamma\to \RR^{*}{\cdot} 1) & \cong & \{\pm 1, \pm i\}\cong \mu_{4}\,,\\
\Pin'^{-}(1) & =  \ker(N':\Gamma\to \Cl^{*})  & \cong &  \{\pm 1\}\cong \mu_{2}\,,\\
\PIN^{-}(1) & =  \ker(N^{2}:\Gamma\to \RR^{*}{\cdot} 1)   & = & \ker(N:\Gamma\to \RR^{*}{\cdot} 1)\cong \mu_4 \,.
\end{align*}
Here $\mu_{d}$ represents the multiplicative cyclic group of order $d$. In all three cases, intersection with
$S\Gamma$ returns $\{\pm1\}\cong \mu_{2}$. \\
One sees that the big Pin group in this case is the same as $\Pin$ and that $\Pin'$ is strictly smaller. For all three cases we get that $\Spin^{-}(1) \cong \Pin'^{-}(1) \cong \mu_2$.
\end{example}

\begin{example}[$\mathbf{\Cl_{1,0}}$] Next consider $V=\RR e$ with $q(re) = r^{2}$, thus, 
\[
\Cl_{0,1}\cong \RR[e]/(e^{2}-1)\cong \RR[e]/(e-1)\times \RR[e]/(e+1)\cong\RR\times\RR\,,
\] 
the composition sending $a+be\mapsto (a+b, a-b)$. Note that the inverse isomorphism sends
$(c,d)\in\RR\times \RR$ to $\tfrac{1}{2}((c+d) + (c-d)e)$.

Set $\Delta(r) = (r,r), \nabla(r) = (r,-r)$, for $r\in\RR$.
The even Clifford algebra is then $\Delta(\RR)\subset \RR\times \RR$, the odd part is 
$\nabla(\RR) \subset \RR\times\RR$. 

The principal automorphism $\alpha$ corresponds to $(c,d)\mapsto (d,c)$, while 
the principal antiautomorphism $t$ is again the identity on $\RR\times \RR$.

Accordingly, we have 
\[
\begin{array}{|c|c|c|c|c|c|c|c|c|c|}
\hline
\Cl^{\uzero}&\Cl^{\uone}&\Cl^{*}&\Cl^{*}_{hom}&\Gamma&S\Gamma&O& N(c,d)&N'(c,d)\\
\hline
\Delta(\RR)&\nabla(\RR)&\RR^{*}\times\RR^{*}&\Delta(\RR^{*})\cup\nabla(\RR^{*})& \Cl^{*}_{hom}&
\Delta(\RR^{*})& \{\pm1\}& cd(1,1)&(c^{2}, d^{2})\\
\hline
\end{array}
\]
Note again that $\Gamma = (\Cl_{1,0})^{*}_{hom}$ as $\Cl_{1,0}$ is commutative, thus, $(-1)^{|u||x|}uxu^{-1} =
(-1)^{|u||x|}x$ and preservation of $i_{q}V$ is automatic.

For use below, we record as well that $(Z^{\uzero})^{*}=\Delta(\RR^{*})$ so that
\begin{align*}
\Pin^+(1) & = \ker(N:\Gamma\to \Delta(\RR^{*})) &\cong & \{\pm (1,1)\}\cong \mu_{2}\,,\\
\Pin'^{+}(1) & = \ker(N':\Gamma\to \Cl^{*}) &\cong & \{(\pm 1,\pm 1)\}\cong \mu_{2}\times \mu_{2}\,,\\
\PIN^{+}(1) & = \ker(N^{2}:\Gamma\to \Delta(\RR^{*}))&= & \ker(N':\Gamma\to \Cl^{*})\,.
\end{align*}
In all three cases, intersection with $S\Gamma$ returns $\{\pm(1,1)\}\cong \mu_{2}$. Here we have $\Pin'^{+}(1)\cong \PIN^{+}(1) \cong \mu_2 \times \mu_2$ and $\Pin^{+}(1)\cong  \Spin^{+} \cong \mu_2$.
\end{example}

\begin{example}[$\mathbf{\Cl_{0,2}}$] Here $V=\RR e_1 \oplus \RR e_2$ with $q(x_1e_1 + x_2e_2)=-x_1^2-x_2^2$. As algebra we have $\Cl_{0,2}\cong \HH$. The isomorphism is given as $1 \mapsto1$, $e_1 \mapsto i$, $e_2 \mapsto j$, and  $e_1e_2=e_{12} \mapsto k$. We also have the embedding $\lambda(a +bi)=a +bk$.
\[
\begin{array}{|c|c|c|c|c|c|c|c|c|c|}
\hline
\Cl^{\uzero}&\Cl^{\uone}&\Cl^{*}&\Cl^{*}_{hom}&\Gamma&S\Gamma&O\\
\hline
\lambda(\CC)=\RR \oplus \RR k & \RR i \oplus \RR j & \HH^* & \lambda( \CC^*) \cup \lambda( \CC^*)i  & \Cl^*_{hom} & \lambda(\CC^*) & S^1 \cup S^1  \\
\hline
\end{array}
\]
The norms are given as $N: \Gamma \xrightarrow{} (Z^{\underline{0}})^*$, $N': \Gamma \xrightarrow{} \Cl^*_{hom}$. Note that $(Z^{\underline{0}})^*=\RR^* \cdot 1$.
\[
\begin{array}{|c|c|c|c|c|}
\hline
\mathrm{norm} & x_0 + x_3k&x_1i + x_2j&\ker(\mathrm{norm}: \Gamma \xrightarrow{} (Z^{\underline{0}})^*)& \ker(\mathrm{norm}: \Gamma \xrightarrow{} (Z^{\underline{0}})^*)  \cap S\Gamma  \\
\hline
N &  x_0^2 +x_3^2& x_1^2-+x_2^2 & \lambda(S^1)\cup \lambda(S^1)i & \lambda(S^1)  \\
N' & x_0^2+x_3^2 & -x_1^2 - x_2^2 & \lambda(S^1) & \lambda(S^1)  \\
N^2=(N')^2  & (x_0^2+x_3^2)^2 &(x_1^2 + x_2^2)^2& \lambda(S^1) \cup \lambda(S^1)i  & \lambda(S^1)  \\
\hline
\end{array}
\]
This means that $\Pin^-(2)\cong \PIN^-(2)\cong \lambda(S^1) \cup \lambda(S^1)i$ and $\Pin'^{-}\cong  \Spin^-(2)\cong S^1$.
\end{example}

\begin{example}[$\mathbf{\Cl_{2,0}}$] Here $V=\RR e_1 \oplus \RR e_2$ with $q(x_1e_1 + x_2e_2)=x_1^2+x_2^2$. As algebra we have $\Cl_{2,0}\cong \Mat_{2 \times 2}(\RR)$. The isomorphism is given as 
$$1 \mapsto {\bf 1}, \  e_1 \mapsto \begin{pmatrix} 0 & 1 \\ 1 & 0 \end{pmatrix}, \ e_2 \mapsto \begin{pmatrix}  -1 & 0 \\ 0 & 1 \end{pmatrix}, \ \textrm{and } e_{12} \mapsto \begin{pmatrix}   0 & 1 \\ -1 & 0 \end{pmatrix}. $$
 Consider the embedding $\lambda: \CC \xrightarrow{} \Mat_{2 \times 2}(\RR), (x+yi) \mapsto \begin{pmatrix}  x & y \\ -y & x \end{pmatrix}$.

\[
\begin{array}{|c|c|c|c|c|c|c|c|c|c|}
\hline
\Cl^{\uzero}&\Cl^{\uone}&\Cl^{*}&\Cl^{*}_{hom}&\Gamma&S\Gamma&O\\
\hline
\lambda(\CC) & \lambda(\CC)e_1 & \GL(2,\RR) & \lambda( \CC^*) \cup \lambda( \CC^*)e_1  & \Cl^*_{hom} & \lambda(\CC^*) & S^1 \cup S^1  \\
\hline
\end{array}
\]
The norms are given as $N: \Gamma \xrightarrow{} \lambda(\CC^*)$, $N': \Gamma \xrightarrow{} \Cl^*_{hom}$. Note that $(Z^{\underline{0}})^*=\RR^* \cdot 1$.
\[
\begin{array}{|c|c|c|c|c|}
\hline
\mathrm{norm} & \lambda(x_0 + x_3i)&\lambda(x_1 + x_2i)e_1&\ker(\mathrm{norm}: \Gamma \xrightarrow{} (Z^{\underline{0}})^*)& \ker(\mathrm{norm}: \Gamma \xrightarrow{} (Z^{\underline{0}})^*) \cap S\Gamma  \\
\hline
N &  x_0^2 +x_3^2& -x_1^2-x_2^2 & \lambda(S^1) & \lambda(S^1)  \\
N' & x_0^2+x_3^2 & x_1^2 + x_2^2 & \lambda(S^1) \cup \lambda(S^1)e_1 & \lambda(S^1)  \\
N^2=(N')^2  & (x_0^2+x_3^2)^2 &(x_1^2 + x_2^2)^2& \lambda(S^1) \cup \lambda(S^1)e_1  & \lambda(S^1)  \\
\hline
\end{array}
\]
Here we get $\Pin^+(2)\cong \Spin^+(2)\cong S^1$ and $\Pin'^+(2)\cong \PIN(2)=\lambda(S^1) \cup \lambda(S^1)e_1$.
\end{example}


\begin{example}[$\mathbf{\Cl_{0,3}}$] Next consider $V=\oplus_{i=1}^{3}\RR e_{i}$ with 
$q(\sum_{i=1}^{3}x_{i}e_{i}) = -(x_{1}^{2}+x_{2}^{2}+x_{3}^{2})$. Using the table of Theorem \ref{Thm:realClifftable} one sees that 
$\Cl_{0,3}=\HH \times \HH$. The calculations of the Pin and Spin groups are analogous as to the case $\Cl_{3,0}$ that was treated in Section \ref{Sec:Pin3}. The readers are invited to carry out the details themselves. The results are $\PIN^{-}(3) \cong \Pin^{-}(3) \cong  \Un(2)^{\pm 1}$ and $\Pin'^{-}(3)\cong \Spin^{-}(3) \cong \SU(2)$.
\end{example}

\section{Acknowledgements}

The authors thank the Mathematisches Forschungsinstitut Oberwolfach for the inspiring perfect working environment. \\
E.F. and C.I. want to thank Paul Mezo for helpful discussions, and in particular we want to thank Ruth Buchweitz for her help and hospitality.


\newcommand{\etalchar}[1]{$^{#1}$}
\def\cprime{$'$}

\end{document}